\documentclass[a4paper,12pt]{amsart}
\usepackage{txfonts}
\usepackage[T1]{fontenc}
\usepackage[latin1]{inputenc}
\usepackage{graphicx}
\usepackage{amsmath}
\usepackage[frenchb]{babel}
\usepackage{amssymb,hyperref}
\theoremstyle{theorem}
\newtheorem{theo}{Théorème}[section]
\newtheorem{prop}[theo]{Proposition}
\newtheorem{lemm}[theo]{Lemme}
\newtheorem{maintheo}{Théorème}

\theoremstyle{definition}
\newtheorem{defi}[theo]{Définition}
\newtheorem{rema}[theo]{Remarque}
\newtheorem{conv}{Convention\!\!}

\DeclareMathAlphabet{\mathrmsl}{OT1}{cmr}{m}{sl}
\newcommand{\cA}{\mathcal{A}}
\newcommand{\cC}{\mathcal{C}}
\newcommand{\cD}{\mathcal{D}}
\newcommand{\cH}{\mathcal{H}}
\newcommand{\cM}{\mathcal{M}}
\newcommand{\cL}{\mathcal{L}}
\newcommand{\cQ}{\mathcal{Q}}
\newcommand{\cR}{\mathcal{R}}

\newcommand{\II}{\mathbb{I}}
\renewcommand{\leq}{\leqslant}
\renewcommand{\geq}{\geqslant}

\newcommand{\NM}{{\mathbb N}}

\newcommand{\setR}{{\mathbb R}}
\newcommand{\setC}{{\mathbb C}}
\newcommand{\setN}{{\mathbb N}}

\newcommand{\dA}{\dot{\mathcal{A}}}

\DeclareMathOperator{\Ric}{Ric}
\DeclareMathOperator{\Tr}{tr}
\newcommand{\vol}{\operatorname{vol}}

\newcommand{\Ker}{\operatorname{Ker}}

\newcommand{\nabw}{\widetilde{\nabla}}

\renewcommand{\geq}{\geqslant}
\renewcommand{\leq}{\leqslant}

\newcommand{\eps}{\varepsilon}
\renewcommand{\exp}{\operatorname{exp}}

\newcommand{\opI}{\mathcal{I}}

\newcommand{\Ltrans}{L^{\textrm{trans}}}
\newcommand{\nabtrans}{\nabw^{\textrm{trans}}}
\newcommand{\nabemphtrans}{\nabw^{\textrm{\emph{trans}}}}

\renewcommand{\Re}{\operatorname{Re}}
\renewcommand{\Im}{\operatorname{Im}}

\newcommand{\widebar}{\bar}
\newcommand{\proofof}[1]{\end{#1}\begin{proof}}

\newcounter{mnotecount}[section]
\renewcommand{\themnotecount}{\thesection.\arabic{mnotecount}}
\newcommand{\mnote}[1]
{\protect{\stepcounter{mnotecount}}$^{\mbox{\footnotesize  $
      \bullet$\themnotecount}}$ \marginpar{\raggedright\tiny\em
    $\!\!\!\!\!\!\,\bullet$\themnotecount: #1} }

\swapnumbers

\begin{document}
\title[Demi-espace hyperbolique et poly-homogénéité]{Analyse sur un
  demi-espace hyperbolique et poly-homogénéité locale}
\author{O. Biquard} \address{UPMC Université Paris 6 \\ UMR 7586 CNRS, Institut de mathématiques de Jussieu}
\author{M. Herzlich} \address{ Université Montpellier 2 \\ UMR 5149 CNRS, Institut de mathématiques et de modélisation de 
Montpellier} 

\maketitle

\section*{Introduction}
L'objectif de cet article est de donner une preuve de la polyhomogénéité des métriques d'Einstein asymptotiquement
hyperboliques réelles ou complexes dans le cas local. Ces métriques, introduites par Fefferman et Graham dans le cas réel 
au milieu des années 80 \cite{FefGra85}, se sont révélées être des instruments essentiels d'étude de la géométrie conforme ou 
de la géométrie CR~: une métrique asymptotiquement hyperbolique induit en effet sur son \emph{bord à l'infini} une structure
conforme ou CR. Réciproquement, des théo­rèmes d'existence de métriques d'Einstein asymptotiquement hyperboli­ques avec
structure à l'infini fixée ont été donnés ces dernières années \cite{And05,Biq00,GraLee91}. 

L'idée principale de la théorie est que l'équation d'Einstein rigidifie suffisamment la situation pour que les propriétés de
la structure placée à l'infini se reflète fidèlement dans la géométrie de la métrique asymptotiquement hyperbolique. Un outil
essentiel de cette correspondance est l'existence d'un développement en série de ces métriques d'Einstein
au voisinage de l'infini d'un type particulier, dit \emph{développement polyhomogène} (voir
la section \ref{sec:six} pour une définition précise).
Les premiers termes de ce développement sont en général \emph{formellement déterminés} (c'est-à-dire calculables
à partir d'un jet d'ordre fini de la structure à l'infini) et servent de briques de base pour la construction
d'invariants conformes ou CR. Les termes suivants dépendent de la géométrie globale de $g$ et sont donc \emph{formellement 
indéterminés} mais les premiers d'entre eux jouent un rôle crucial dans nombre de questions analytiques (voir plus bas).

Dans le cas hyperbolique réel, nos résultats étendent ceux de
Chru\'sciel, Delay, Lee et Skinner \cite{ChrDelLeeSki05} qui ont prouvé
la polyhomogénéité dans le cas global, c'est-à-dire quand le bord à
l'infini est une variété fermée (compacte sans bord), et de Helliwell
\cite{Hel08} qui a donné une preuve de la polyhomogénéité dans le cas
local en dimension paire. Les résultats nouveaux sont donc, d'une part le
cas complexe, d'autre part le cas réel local en dimension impaire.

Notre premier résultat peut s'énoncer comme suit (pour un énoncé plus
précis, le lecteur est invité à se référer au théorème
\ref{th:mainphg}).

\begin{maintheo}\label{th:1}
  Soit $M^n$ une variété difféomorphe à une demi-boule dont
  l'intérieur est muni d'une métrique d'Einstein $g$, asymptotiquement
  hyperbolique et dont la métrique induite sur le bord à l'infini est
  lisse. Alors la métrique $g$ possède un développement polyhomogène
  au voisinage de son bord à l'infini.
\end{maintheo}

Comme conséquence intéressante, toute métrique d'Einstein
asymptotiquement hyperbolique complexe satisfait les hypothèses de
\cite{BM} et peut donc être choisie comme facteur d'un produit de
variétés d'Einstein à déformer.

Une fois l'existence établie, la forme du développement polyhomogène a
été précisée par Fefferman et Graham \cite{FefGra85,Gra00,FG} 
dans le cas réel, alors que Seshadri \cite{Ses09} en a étudié les
premiers termes dans le cas complexe. L'information essentielle est l'existence de
coordonnées privilégiées au voisinage de l'infini, dans lesquelles le développement possède une structure
particulière : un voisinage de l'infini est écrit comme $(A,+\infty)×\partial M$,
et cette décomposition peut être faite de sorte que les coordonnées
soient géodésiques, c'est-à-dire de sorte que la métrique s'écrive
comme un produit tordu
$$ g = dr^2 + g_r , $$
où $g_r$ est une famille de métriques sur $\partial M$.  Dans le cas réel, le
développement s'écrit, quand $r\to+\infty$,
$$ g_r = e^{2r}\big(\gamma + \gamma_2 e^{-2r} + \cdots + \gamma_{n-2}e^{-(n-2)r} + h re^{-(n-1)r} + \gamma_{n-1}e^{-(n-1)r} + \cdots \big) $$
où $\gamma$ est une métrique sur le bord à l'infini, et $\gamma_2$, \dots, $\gamma_{n-2}$ et $h$ 
sont des champs de formes quadratiques sur l'espace tangent du bord à l'infini for­mel­lement déterminés, c'est-à-dire 
calculables comme des polynômes en un jet d'ordre fini de $\gamma$ et de son inverse. Le premier terme formellement 
indéterminé $\gamma_{n-1}$ est également un champ de formes quadratiques sur le bord à l'infini, et il pilote la suite 
du développement au sens où sa valeur, jointe à celle de 
$\gamma$, \dots, $\gamma_{n-2}$ et $h$, détermine celle de tous les termes suivants. 
La métrique $g$ est quasi-isométrique à son terme principal $g_0$, donné par
$$ g_0 = dr^2 + e^{2r}\gamma . $$

Dans le cas complexe, des coordonnées similaires existent : le
terme principal est donné, quand $r\to+\infty$, par
$$g_0=dr^2+\frac14 e^{2r}\eta^2+\frac14 e^r\gamma,$$
où $\eta$ est une 1-forme de contact sur le bord et $\gamma$ une métrique sur
la distribution de contact, induite par une structure CR. Alors existe
un développement similaire en les puissances de $e^{-r/2}$,
$$ g = g_0 + g_1 e^{-r/2} + g_2 e^{-r} + \cdots , $$
dans lequel le premier terme indéterminé \cite{Biq00} arrive à l'ordre
$n=2m$ (c'est-à-dire au terme d'ordre $e^{-mr}$). Seshadri \cite{Ses09} a mis en évidence deux termes
logarithmiques possibles dans le développement, dont le plus élevé, correspondant
au terme logarithmique de Lee-Melrose \cite{LeeMel82} pour une métrique de
Kähler-Einstein, apparaît à l'ordre $n+2$ (c'est-à-dire au terme d'ordre $e^{-(m+1)r}$). 

\smallskip

Une application de ces résultats est un résultat de continuation
unique pour les métriques d'Einstein asymptotiquement hyperbolique
réelles ou complexes.

\begin{maintheo}\label{th:2}
  Soit $M^n$ une variété difféomorphe à une demi-boule dont
  l'intérieur est muni de deux métriques d'Einstein asymptotiquement
  hyperboliques $g_1$ et $g_2$. Si on a $|g_1-g_2|_{g_1} =
  o(e^{-(n-1)r})$ dans le cas réel, ou $|g_1-g_2|_{g_1}
  =o(e^{-(m+1)r})$ dans le cas complexe ($n=2m$), alors il existe un
  difféomorphisme $\Phi$ de la demi-boule préservant point par point le
  bord à l'infini tel que $g_2=\Phi^*g_1$.
 \end{maintheo}
Dans le cas réel, l'hypothèse de coïncidence jusqu'à l'ordre $n-1$ se
résume à la coïncidence des structures conformes $\gamma$ et du terme
indéterminé $g_{n-1}$. Par polyhomogénéïté (théorème \ref{th:1}), on a
alors une coïncidence à un ordre infini en jauge géodésique :
$|g_1-g_2| =O(e^{-\infty r})$ . Les résultats de \cite{Biq08} impliquent
alors $g_1=g_2$, toujours en jauge géodésique, donc le théorème
\ref{th:2} ne nécessite pas de démonstration dans ce cas.

Dans le cas complexe, le nombre et la nature des termes indéterminés
reste un peu mystérieuse, mais l'hypothèse de coïncidence jusqu'à
l'ordre $n+2$ assure l'égalité de tous les termes indéterminés de
$g_1$ et $g_2$, et donc $|g_1-g_2| =O(e^{-\infty r})$ dans une jauge
géodésique. À partir de cette coïncidence à un ordre infini est
employée une méthode similaire à celle de \cite{Biq08}, mais
compliquée en raison de la complexité de la géométrie hyperbolique
complexe d'une demi-boule.

\smallskip

L'article est organisé comme suit~: dans un premier temps, nous
mettons en place des éléments d'analyse dans un demi-espace
hyperbolique. Le cadre géométrique est clair dans le cas réel, mais
certaines formules utiles semblent être absentes de la littérature
dans le cas complexe, la première section de cet article leur est donc
consacrée.  Dans la deuxième section, on introduit des espaces à
double poids adaptés à la géométrie du demi-espace, nécessaires à
cause des coins du bord, et on les utilise pour démontrer des
théorèmes d'isomorphisme pour des opérateurs du type du laplacien.  La
troisième section introduit les métriques asymptotiquement
hyperboliques (d'Einstein ou non) dans le cadre local que nous
souhaitons considérer, et les décrit dans un système de coordonnées
adaptées à la géomé­trie au voisinage de l'infini. La quatrième
section place ces métriques dans une jauge locale adaptée vis-à-vis
des difféomorphismes~; c'est ici que l'ana­lyse développée
précédemment s'introduit naturellement. La cinquième section étend
cette analyse à toutes les métriques asymptotiquement hyperboliques,
extension qui est immédiatement utilisée dans la sixième section pour
démontrer un lemme technique essentiel de décroissance des dérivées
transverses des éléments du noyau des opérateurs du type du laplacien.
La septième section contient la preuve de la polyhomogénéité. La
section 8, enfin, contient la preuve de la continuation unique.

\section{Un modèle du demi-espace hyperbolique complexe}
\label{sec:un}
Considérons tout d'abord, à titre d'exemple, un demi-espace
hyperbo­lique réel $\setR H^n_+\subset\setR H^n$, délimité par un hyperplan
totalement géodésique $\setR H^{n-1}\subset\setR H^n$, espace fixe d'une inversion
$\iota$. Dans le modèle du demi-espace supé­rieur $\{ x_1>0 \}$, avec
métrique $\frac{dx_1^2+\cdots+dx_n^2}{x_1^2}$, on peut choisir l'inver­sion
$\iota:x_n\mapsto-x_n$, donc $\setR H^n_+=\{x_n > 0\}$, et la métrique de $\setR H^n$
s'écrit sous la forme agréable
\begin{equation}
  \label{eq:23}
  g_h = dr^2 + \cosh^2(r) \gamma ,
\end{equation}
où $r$ est la distance (avec signe) au bord $\setR H^{n-1}=\{ x_n=0\}$ et $\gamma$ est la
métrique hyperbolique sur ce bord $\setR H^{n-1}$. Dans
ces coordonnées, l'inver­sion $\iota$ est $r\mapsto -r$.

La situation est considérablement plus compliquée dans le cas
hyperbolique complexe, car il n'y a pas d'hyperplan totalement
géodésique. Classiquement, on choisit à la place un bisecteur, qui est
un hyperplan minimal.

Plus précisément, l'espace hyperbolique complexe peut se décrire comme
\begin{equation}
  \label{eq:1}
  \setC H^m = \{ [x]\in \setC P^m, \langle x,x\rangle < 0 \}
\end{equation}
pour le produit hermitien lorentzien $\langle x,x'\rangle = 4 \Re(\bar x_0 x'_m) +
\sum_1^{m-1} \bar x_j x'_j $.  La métrique hyperbolique complexe, à
courbure sectionnelle comprise entre $-1$ et $-\frac14$, est alors
donnée au point $x\in \setC H^m$ par
\begin{equation}
  \label{eq:2}
  (g_h)_x(X,X) = 4 \frac{\langle x,x\rangle\langle X,X\rangle-|\langle x,X\rangle|^2}{-\langle x,x\rangle^2} .
\end{equation}
Soit $f:\setC H^m\to\setR_+$ définie par $f(x)=-\frac{\langle x,x\rangle}{4|x_0|^2}$. Dans la
carte affine $z_0=-1$, on peut la voir comme la fonction $f:\setC^m\to\setR$
obtenue par $f(z_1,\dots,z_m)=\Re(z_m) - \frac14 \sum_1^{m-1}
|z_i|^2$. Cela mène au modèle du demi-espace de Siegel, dans lequel
l'espace hyperbolique complexe est le domaine
\begin{equation}
  \label{eq:3}
  \setC H^m = \{  f(z_1,\dots,z_m) > 0 \} ,
\end{equation}
et sa métrique s'écrit
\begin{equation}
  \label{eq:5}
  g = \frac{df^2+\eta^2}{f^2}+\frac{|dz_1|^2+\cdots +|dz_{m-1}|^2}f ,
\end{equation}
où l'on a posé $z_m = f+\frac14 \sum_1^{m-1} |z_i|^2 - i v$ et
\begin{equation}
  \label{eq:6}
  \eta = dv + \frac12 \Im(\bar z_1 dz_1 + \cdots + \bar z_{m-1} dz_{m-1}) .
\end{equation}

On considère le demi-espace hyperbolique complexe,
\begin{equation}
  \label{eq:7}
  \setC H^m_+ = \{ \Im(z_m) < 0 \} = \{ v>0 \} ,
\end{equation}
dont le bord est le bisecteur
\begin{equation}
  \label{eq:8}
  B = \{ \Im(z_m)=0 \} .
\end{equation}
Ce bisecteur est l'hypersurface équidistante des deux points
$(0,± z_m)$, pour un $z_m\in \setR$. La géodésique
\begin{equation}
  \label{eq:19}
  \{ (0,iz_m), z_m\in \setR \}
\end{equation}
engendre l'\emph{épine complexe} du bisecteur, à savoir la géodésique complexe
\begin{equation}
  \label{eq:9}
  \Sigma=\{ (0,z_m), \Re(z_m)>0 \},
\end{equation}
qui intersecte le bisecteur en son \emph{épine}, la géodésique (réelle)
\begin{equation}
  \label{eq:10}
  \sigma = \Sigma\cap B = \{ (0,z_m), z_m\in \setR_+^* \} .
\end{equation}

Le but de cette section est de donner une expression de la métrique
du demi-espace hyperbolique complexe en termes de la distance et de la
projection sur le bisecteur (nous n'avons pas trouvé cette formule dans 
la littérature). Prenons sur le bisecteur des coordonnées
$(\tau,\rho,y)\in \setR×\setR_+^*× S^{2m-3}$, telles que
\begin{equation}
  \label{eq:11}
  \big(((z_1,\dots,z_{m-1}),z_m\big)
 = \big(2e^{\frac \tau2}\tanh(\tfrac \rho2)y,e^\tau\big) .
\end{equation}
Soit $\theta=J\frac{d\rho}\rho$ la forme de contact standard sur la sphère
$S^{2m-3}$, et $\gamma'$ la restriction de la métrique standard de la sphère
à la distribution de contact $\ker \theta$.

Nous paramétrons donc un point de $\setC H^m_+$ par sa distance $s$ au
bisecteur et sa projection $(\tau,\rho,y)$ sur celui-ci. Posons
\begin{align}
  \vartheta_1 &= \theta + \frac{\tanh \tfrac s2}{2\cosh \tfrac \rho2} d\tau , \label{eq:16}\\
  \vartheta_2 &= (1+\tanh^2 \tfrac s2) d\tau + 2 \frac{\sinh^2 \tfrac \rho2}{\cosh
    \tfrac \rho2} \tanh(\tfrac s2) \theta .\label{eq:17}
\end{align}

\begin{lemm}\label{lemm:met-demi}
  La métrique hyperbolique complexe sur $\setC H^m_+$ s'écrit en fonc­tion de
  la distance au bisecteur et de la projection sur celui-ci,
  \begin{multline}
    \label{eq:12}
    g = ds^2 + \cosh^4(\tfrac s2)\cosh^2(\tfrac \rho2) \vartheta_2^2 \\
 + \cosh^2(\tfrac s2) \Big[ d\rho^2 + \sinh^2(\rho) \vartheta_1^2 + 4 \sinh^2(\tfrac \rho2) \gamma' \Big]
  \end{multline}
\end{lemm}
Si la variable $s$ décrit $\setR$ tout entier, alors les formules
(\ref{eq:16}), (\ref{eq:17}) et (\ref{eq:12}) restent valables et
donnent un modèle de la métrique hyperbolique complexe sur $\setC H^m$
entier. Il y a aussi une inversion holomorphe de $\setC H^m$,
\begin{equation}
  \label{eq:18}
  \iota:(s,\tau)\mapsto(-s,-\tau)  ,
\end{equation}
qui échange les deux demi-espaces.

Sur le bisecteur lui-même ($s=0$), on obtient la formule classique
\begin{equation}
  \label{eq:13}
   \cosh^2(\tfrac \rho2) d\tau^2 + d\rho^2 + \sinh^2(\rho) \theta^2 + 4 \sinh^2(\tfrac \rho2) \gamma' .
\end{equation}
On retrouve l'épine $\sigma$ du bisecteur en faisant $\rho=0$, et les fibres
de la projection sur l'épine --- les niveaux de $\tau$ ---, sont des
hyperplans $\setC H^{m-1}$ totalement géodésiques dans $\setC
H^m$. L'hyperplan $\tau^{-1}(0)$, qui coupe la géodésique (\ref{eq:19}),
est le lieu fixe de l'inversion $\iota$.

Donnons le changement de coordonnées sur $\setC H^m$, prolongeant
(\ref{eq:11}), et permettant de passer à cette formule. On écrit
\begin{equation}
  \label{eq:14}
  \big((z_1,\dots,z_{m-1}),z_m\big) = \big(e^{\frac{\tau+i\alpha}2}ty, e^{\tau+i\alpha}\big) , 
\end{equation}
où $-\tfrac \pi2 < \alpha < \tfrac \pi2$ et $0 \leq t < 2 \sqrt{\cos \alpha}$. Le
changement de variables est alors donné par
\begin{equation}
  \label{eq:15}
  te^{i\frac \alpha2} = \frac{2\sinh \tfrac \rho2}{\cosh \tfrac \rho2+i \tanh \frac s2} .
\end{equation}

Enfin, la structure complexe de $\setC H^m$ est donnée par
\begin{equation}
  \label{eq:4}
  Jdr = \cosh^2(\tfrac s2)\cosh(\tfrac \rho2)\vartheta_2, \quad
  Jd\rho = \sinh(\rho)\vartheta_1 .
\end{equation}

Finalement, observons que la formule (\ref{eq:12}) a un comportement
un peu curieux quand $\rho\to+\infty$, puisqu'alors $\vartheta_1\to\theta$ et $\vartheta_2\sim2\sinh(\frac
\rho2)\theta$, donc la base choisie pour écrire la métrique dégénère. En
remplaçant $\vartheta_1$ par $\vartheta_1-\frac{\cosh \frac \rho2}{2 \sinh^2(\frac
  \rho2)\tanh(\frac s2)} \vartheta_2 \sim \frac{d\tau}{2\cosh \frac \rho2}$ quand $s$ et $\rho$
  tendent vers l'infini, on obtient
  \begin{equation}
    \label{eq:54}
    \begin{split}
      g = ds^2 &+ \cosh^4(\tfrac s2)\cosh^2(\tfrac \rho2) \vartheta_2^2 \\ &+
      \cosh^2(\tfrac s2) \Big[ d\rho^2 + \sinh^2(\tfrac \rho2) (d\tau^2 + 4 \gamma')
      \Big] + O(\frac 1{\cosh^2\tfrac s2}) .
    \end{split}
  \end{equation}
En particulier, cette formule nous dit que, quand $\rho\to+\infty$, les
directions de contact explosent en $\cosh \frac s2$, contrairement à
ce qu'une lecture trop rapide de (\ref{eq:12}) eut indiqué.

{\flushleft\emph{Preuve du lemme \ref{lemm:met-demi}}}. -- Par simplicité, on se restreint au cas de
  la dimension $m=2$, l'extension en dimension supérieure étant
  évidente. Le point crucial est de calculer les géodésiques
  orthogonales au bisecteur. Par rapport aux notations précédentes,
  on posera $y=e^{i\beta}$ puisque $y$ est une variable circulaire si
  $m=2$. Considérons le bisecteur
  \begin{equation}
    \label{eq:20}
    B_s = \{ (\sqrt{s} t e^{i\beta+i\tfrac \alpha2}, se^{i\alpha}), |t|<2\sqrt{\cos
      \alpha} \} . 
  \end{equation}
  Ses méridiens $B_{s,\beta}$ s'obtiennent en fixant $\beta$, et coupent $B$
  en
  \begin{equation}
    \label{eq:21}
    \gamma_{s,\beta} := B\cap B_{s,\beta} = \{ (\sqrt s t e^{i\beta}, s), |t|<2 \} 
  \end{equation}
  qui est une géodésique orthogonale à l'épine $\sigma$ au point
  $(0,s)$. 

\smallskip

{\flushleft\emph{Affirmation}}. Les géodésiques du plan méridien $B_{s,\beta}$ qui
  sont orthogonales à $\gamma$ sont en réalité orthogonales à $B$ tout
  entier. 

\smallskip

  En effet, d'après \cite{Gol99}, si $x$ est un point de l'espace
  hyperbolique, alors sa projection $p(x)=(\sqrt s e^{i\beta} t,r)$ sur
  $B$, sa projection $\pi(x)$ sur l'épine complexe $\Sigma$, et enfin la
  projection $\xi=\pi(p(x))=(0,r)$ de $p(x)$ sur l'épine $\sigma$ engendrent
  un plan totalement réel. Ce plan, coupant le bisecteur $B$ en la
  géodésique $\gamma_{s,\beta}$, et l'épine complexe $\Sigma$ en la géodésique
  orthogonale à $\sigma$ en $\xi$, ne peut être que le méridien $B_{s,\beta}$. Il
  contient donc la géodésique reliant $x$ à $p(x)$.

Chaque plan méridien $B_{s,\beta}$ est muni de coordonnées $(t,\alpha)$. Pour
l'identi­fier à un plan hyperbolique réel, on se ramène à $B_{1,0}$
par l'isométrie 
$$(z_1,z_2)\to(\sqrt{s}e^{i\beta}z_1,s).$$ 
Revenant aux coordonnées homogènes (\ref{eq:1}), on voit que
$$ B_{1,0} = \{[-e^{-i\frac \alpha2}:t:e^{i\frac \alpha2}]\} = \{[-u:1:\bar u]\},$$
avec $u=(te^{i\frac \alpha2})^{-1}$, satisfaisant $\Re u^2>\frac 14$.  C'est
un modèle de Klein du plan hyperbolique réel, dans lequel les
géodésiques orthogonales à $\gamma_{1,0}=\{ u\in \setR\}$ sont des droites
verticales, qu'on paramètre par
$$ u = \frac{\cosh \tfrac \rho2 + i \tanh \tfrac s2}{2\sinh \tfrac \rho2}. $$
Ici, $\rho$ est l'abscisse curviligne sur la géodésique $\gamma_{1,0}$, et $s$
la distance à $\gamma_{1,0}$. Cela donne les coordonnées voulues. Le reste
de la démonstration du lemme consiste en le calcul de la métrique, que
nous ne reproduisons pas ici. \qed

Un bisecteur est une hypersurface minimale. Le lemme suivant montre
que la seconde forme fondamentale de toutes les hypersurfaces
équidistantes d'un bisecteur est uniformément bornée.
\begin{lemm}
  \label{lemm:uniform}
  Les hypersurfaces $s=\mathrm{cst}$, à distance constante du
  bisecteur, ont une seconde forme fondamentale uniformément bornée
  indépendamment de $s$. Il en est de même de chaque dérivée
  covariante de la seconde forme fondamentale.
\end{lemm}
\begin{proof}
  À partir du lemme \ref{lemm:met-demi}, il s'agit d'un calcul
  direct. On écrit $g_h=ds^2+g_s$, alors la seconde forme fondamentale
  est $\II=-\frac12 \partial_sg_s$. On obtient pour $g_s^{-1}\II$ :
  \begin{itemize}
  \item sur $\frac \partial{\partial\rho}$ et sur $\ker \theta\subset TS^{2m-3}$, une valeur propre
    $-\frac 12\tanh \frac s2$ ;
  \item dans la base orthonormale $(\cosh^2(\frac s2)\cosh(\frac
    \rho2)\vartheta_2,\cosh(\frac s2)\sinh(\rho)\vartheta_1)$, une matrice
    \begin{multline*}
      \frac {-1}{2(\cosh^2 \tfrac \rho2 +\tanh^2 \tfrac s2 )}× \\
      \begin{pmatrix}
        \tanh(\tfrac s2) \big(4+\sinh^2(\tfrac \rho2) (3-\tanh^2 \tfrac s2)\big) & 
        \frac{\sinh \rho}{\cosh^3 \tfrac s2} \\
        \frac{\sinh \rho}{\cosh^3 \tfrac s2} & 
        \tanh(\tfrac s2) (1+\tanh^2 \tfrac s2 \cosh^2 \tfrac \rho2)
      \end{pmatrix}
    \end{multline*}
  \end{itemize}
  Le lemme se déduit de ces formules.
\end{proof}

Remarquons qu'on peut déduire du lemme la courbure moyenne des
hypersurfaces, à savoir
\begin{equation}
  \label{eq:22}
  H = \Tr(g_s^{-1}\II)
= - \tanh(\tfrac s2)\big(m+\frac 2{\cosh^2(\tfrac s2)(\cosh^2 \tfrac \rho2+\tanh^2 \tfrac s2)}\big) . 
\end{equation}
On retrouve ainsi que le bisecteur $s=0$ est minimal, mais les autres
hypersurfaces équidistantes du bisecteur ne sont plus minimales ou à
courbure moyenne constante. Il serait intéressant de trouver le
feuilletage de $\setC H^m_+$ par des hypersurfaces à courbure moyenne
constante.

\section{Analyse sur un demi-espace}
\label{sec:deux}
Dans cette section, on va donner quelques outils simples d'analyse sur
un demi-espace hyperbolique réel ou complexe. Il y a deux bords sur le
demi-espace $\setR H^n_+$ ou $\setC H^m_+$, quand la distance $s$ au bord
intérieur tend vers $0$, ou quand on va à l'infini sur chaque tranche
$s=\mathrm{cst}$. Cela motive l'introduction d'espaces fonctionnels
avec deux poids.

Dans le cas réel, le plus simple, chaque tranche est un espace
hyperbolique réel $\setR H^{n-1}$, sur lequel on choisit des coordonnées
polaires, de sorte que la formule (\ref{eq:23}) devient
\begin{equation}
  \label{eq:24}
  g_h = ds^2 + \cosh^2(s) \big( d\rho^2 + \sinh^2(\rho) g_{S^{n-2}}\big) .
\end{equation}
Il est naturel de considérer la fonction poids
\begin{equation}
  \label{eq:25}
  w = \cosh(s)^{\delta_1} \cosh(\rho)^{\delta_2} ,
\end{equation}
où $\delta_1$ et $\delta_2$ sont deux réels fixés, et de définir pour $k\in\setN$ et $\alpha\in ]0,1[$ les espaces
fonctionnels
\begin{equation}
  \label{eq:26}
  C^k_{\delta_1,\delta_2}(\setR H^n) = \frac 1w C^{k}(\setR H^n), \quad 
C^{k,\alpha}_{\delta_1,\delta_2}(\setR H^n) = \frac 1w C^{k,\alpha}(\setR H^n) .
\end{equation}

Dans le cas complexe, au vu de la formule (\ref{eq:15}), la fonction
$f=\Re(z_m)-\frac 14 \sum |z_i|^2$, qui définit le bord de $\setC H^m$, s'écrit
\begin{equation}
  \label{eq:28}
  f = \frac{e^{\tfrac \tau2}}{\cosh^2(\tfrac s2)(\cosh^2\tfrac
    \rho2+\tanh^2\tfrac s2)}, 
\end{equation}
et il est ainsi naturel de choisir un poids
\begin{equation}
  \label{eq:27}
  w = \cosh(\tfrac s2)^{2\delta_1} \big( \cosh^2(\tfrac \rho2) \cosh(\tfrac \tau2) \big)^{\delta_2} 
\end{equation}
pour définir, de la même manière, les espaces à poids
\begin{equation}
  \label{eq:29}
C^{k}_{\delta_1,\delta_2}(\setC H^m) = \frac 1w C^{k}(\setC H^m) ,\quad  
C^{k,\alpha}_{\delta_1,\delta_2}(\setC H^m) = \frac 1w C^{k,\alpha}(\setC H^m) .
\end{equation}

Une propriété importante des poids choisis, dans les deux cas, est
leur invariance par inversion,
\begin{equation}
  \label{eq:30}
  \iota^*w = w .
\end{equation}

\begin{lemm}\label{lemm:est-poids}
  Supposons que
  \begin{itemize}
  \item dans le cas réel, $0<\delta_1<n-1$ et $0\leq \delta_2\leq n-2$ ;
  \item dans le cas complexe, $0<\delta_1<m$ et $0\leq \delta_2\leq m-\frac 12$
(et $\delta_2\leq\frac 54$ si $m=2$).
  \end{itemize}
  Alors il existe une constante $c(\delta_1)>0$, dépendant de $\delta_1$
  seulement, telle que
  \begin{equation}
 - \Delta\log w - |d\log w|^2 > c(\delta_1) .\label{eq:31}
\end{equation}

\end{lemm}

Dans ce lemme, l'intervalle sur $\delta_1$ est optimal, mais on n'a pas
essayé d'obtenir un intervalle optimal pour le poids $\delta_2$. La
démonstration montrera que, pour chaque valeur de $\delta_1$, on peut
admettre des valeurs de $\delta_2$ qui s'écartent de l'intervalle indiqué
dans le lemme.

\begin{proof}
  Commençons par le cas réel. Écrivons $w=e^u$, où 
  $$ u=\delta_1\log\cosh s+\delta_2\log\cosh \rho=:u_1+u_2.$$ 
  Il faut calculer $-\Delta u-|du|^2$. Puisque
  $ds$ et $d\rho$ sont orthogonaux, on a
  $$ -\Delta u-|du|^2 = -\Delta u_1-|du_1|^2 -\Delta u_2-|du_2|^2 , $$
  et on calcule :
  \begin{align*}
    -\Delta u_1-|du_1|^2 &= \delta_1 \big( 1+(n-2-\delta_1)\tanh^2 s \big) , \\
    -\Delta u_2-|du_2|^2 &= \frac{\delta_2}{\cosh^2 s} \big( 1+(n-3-\delta_2)\tanh^2 \rho \big).
  \end{align*}
  Le lemme s'en déduit immédiatement.

  Le cas complexe est plus compliqué, mais heureusement $ds$, $d\rho$ et
  $d\tau$ sont orthogonaux, et en outre $\Delta\tau=0$. Cela simplifie les
  calculs : on pose $u_1=2\delta_1\log \cosh \frac s2$, $u_2=2\delta_2\log\cosh\frac
  \rho2$ et $u_3=\delta_2\log\cosh\frac \tau2$, alors, si on pose $p=\tanh^2 \frac
  s2$ et $\varpi=\tanh^2 \frac \rho2$, on obtient :
  \begin{align*}
    -\Delta u_1-|du_1|^2 &= \delta_1 \left(
  \tfrac 12 + (m-\tfrac 12-\delta_1)p + \frac{p(1-p)(1-\varpi)}{1+p(1-\varpi)}
                         \right),\\
    -\Delta u_2-|du_2|^2 &= \frac{\delta_2}{\cosh^2\tfrac s2} \left(
  m-1+(\tfrac 12-\delta_2)\varpi -\frac{p\varpi(1-\varpi)}{1+p(1-\varpi)}     \right),\\
    -\Delta u_3-|du_3|^2 &= -\frac{\delta_2^2}{4\cosh^2\tfrac s2}
                      \frac{(1-\varpi)(1-p(1-\varpi))}{(1+p(1-\varpi))^2}.
  \end{align*}
  Puisque $p$ et $\varpi$ prennent leurs valeurs dans $[0,1[$, la formule
  sur $u_1$ donne immédiatement la condition sur le poids $\delta_1$. De
  l'équation sur $u_3$ nous déduisons
$$ -\Delta u_3-|du_3|^2 \geq -\frac{\delta_2^2(1-\varpi)}{4\cosh^2\tfrac r2} , $$
  d'où résulte, posant $v=u_2+u_3$,
  \begin{align*}
    -\Delta v-|dv|^2 &\geq \frac {\delta_2}{\cosh^2\tfrac r2}
    \Big( m-1+(\tfrac 12-\delta_2)\varpi-\varpi(1-\varpi)-\tfrac 14 \delta_2(1-\varpi) \Big) \\
               &\geq \frac {\delta_2}{\cosh^2\tfrac r2}
    \Big( (m-\tfrac 12-\delta_2) - (1-\varpi)(\varpi+\tfrac 12-\tfrac 34 \delta_2) \Big)
      \end{align*}
Si $m\geq 3$ et $0\leq \delta_2\leq m-\frac 12$, cette quantité est positive ; si
$m=2$, elle est positive en se restreignant à $\delta_2\leq \frac 54$ (nous n'avons
fait ici aucun effort pour obtenir la meilleure borne).
\end{proof}

\begin{rema}
  Posons $\cH=n-1$ dans le cas réel, $\cH=m$ dans le cas complexe,
  donc $\cH$ est, au signe près, la limite de la courbure moyenne des
  hypersurfaces de niveau de $s$ quand $s$ tend vers l'infini. On
  observera que dans chaque cas, on a
  \begin{equation}
    \label{eq:37}
    \lim_{r\to\infty} -\Delta\log w-|d\log w|^2 = \delta_1(\cH-\delta_1) .
  \end{equation}
  En revanche, la constante $c(\delta_1)$ ne peut pas être prise égale à
  cette limite.
\end{rema}

L'existence d'un poids vérifiant l'inégalité différentielle
(\ref{eq:31}) permet de déduire immédiatement le comportement du
laplacien dans les espaces à poids $C^{k,\alpha}_{\delta_1,\delta_2}$. En effet,
observons que, pour toute fonction $f$, on a l'identité
\begin{equation}
  \label{eq:32}
  w \Delta f = \Delta(wf) + (-\Delta\log w-|d\log w|^2) wf + 2\langle d\log w,d(wf)\rangle .
\end{equation}
Nous noterons maintenant génériquement $H$ l'espace hyperbolique tout entier (qui est donc, selon 
le cas considéré, soit l'espace hyperbolique réel $\mathbb{R}H^n$ soit l'espace hyperbolique complexe 
$\mathbb{C}H^m$) et $H_+$ le demi-espace avec la métrique calculée dans la section précédente.
Si une fonction $f$ est globalement définie sur $H$, ou bien si
$f$ est définie seulement sur un demi-espace hyperbolique $H_+$ et satisfait
la condition de Dirichlet $f=0$ sur le bord du demi-espace, et si les
poids $\delta_1$ et $\delta_2$ satisfont les conditions du lemme
\ref{lemm:est-poids}, alors, par principe du maximum,
\begin{equation}
  \label{eq:33}
  \sup |wf| \leq \frac 1{c(\delta_1)} \sup |w\Delta f| .
\end{equation}
Par régularité elliptique, on en déduit immédiatement le lemme suivant.
\begin{lemm}\label{lemm:laplacien-scalaire}
  Supposons que les poids $\delta_1$ et $\delta_2$ satisfassent les conditions
  du lemme \ref{lemm:est-poids}. Alors, si $k\geq 2$,
  \begin{enumerate}\item 
    le laplacien, avec condition de Dirichlet, est un isomorphisme
    $$ ^0C^{k,\alpha}_{\delta_1,\delta_2}(H_+) \to C^{k-2,\alpha}_{\delta_1,\delta_2}(H_+),$$ 
    où $ ^0C$ dénote l'espace avec condition de Dirichlet sur le bord intérieur ;
  \item le laplacien sur $H$ entier est un isomorphisme
   $$C^{k,\alpha}_{\delta_1,\delta_2}(H) \to C^{k-2,\alpha}_{\delta_1,\delta_2}(H).$$
  \end{enumerate}
Pour $\lambda\geq 0$, les mêmes résultats restent vrais pour l'opérateur $\Delta+\lambda$,
sous la même condition sur $\delta_2$, et pourvu que 
\begin{equation}
\delta_1\in\left]\tfrac{\cH}2-\sqrt{\tfrac{\cH^2}4+\lambda},\tfrac{\cH}2+\sqrt{\tfrac{\cH^2}4+\lambda}\right[.\label{eq:39}
\end{equation}
\end{lemm}

À nouveau, mentionnons que dans cet énoncé, aucun effort particulier
n'est fait pour obtenir un intervalle optimal pour le poids $\delta_2$.

\smallskip

\begin{proof}
  Nous avons déjà montré la première partie du lemme. Reste à
  démontrer l'énoncé concernant l'opérateur $\Delta+\lambda$. On voit
  immédiatement que si $\lambda\geq 0$, alors la fonction $ -\Delta\log w-|d\log w|^2+\lambda $
  reste minorée par une constante strictement positive, pourvu que
  $\delta_1$ reste dans l'intervalle prescrit par (\ref{eq:39}), et sous la
  même condition sur $\delta_2$. Le même principe de maximum donne alors
  immédiatement le résultat.
\end{proof}

\begin{rema}\label{rema:lambda-negatif}
  Dans le cas $\lambda<0$, il n'est plus toujours vrai que la fonction
  $-\Delta\log w-|d\log w|^2+\lambda$ reste positive sous les conditions 
  du lemme \ref{lemm:est-poids}~:
  elle ne l'est que pour $r$ assez grand. Le même résultat
  d'isomorphisme reste néanmoins vrai, avec une démonstration plus
  compliquée, passant par une première estimation globale plus faible,
  puis une seconde estimation à l'infini en utilisant la positivité
  asymptotique de $-\Delta\log w-|d\log w|^2+\lambda$. Nous ne donnons pas de détail
  car nous n'utiliserons pas ce résultat dans l'article.
 \end{rema}

 Examinons maintenant des problèmes d'analyse sur des systèmes. Le
 cadre général sera celui d'opérateurs géométriques du type
 $L=\nabla^*\nabla+\cR_0$ sur un fibré tensoriel $E$, où $\cR_0$ est un terme de
 courbure, donc à coefficients constants dans toute base orthonormée.
 L'analyse de tels opérateurs sur l'espace hyperbolique réel ou
 complexe est menée dans \cite{Biq00}. Son comportement est gouverné
 par l'opérateur indiciel : si $t$ est la distance à un point, alors
 l'opérateur $L$ «~se comporte~» quand $t$ tend vers l'infini comme
 l'opérateur
\begin{equation}
\mathcal{I} = -\partial_t^2 - \cH \partial_t + \tilde{A} + \cR_0 ,\label{eq:34}
\end{equation}
où $\tilde{A}$ consiste en les termes d'ordre zéro de $\nabla^*\nabla$ (voir
la section \ref{sec:quatre} pour des détails). Si $\lambda$ est la
plus petite valeur propre de $\tilde{A}+\cR_0$, les premiers poids critiques de
(\ref{eq:34}), à savoir les $\mu$ tels que $e^{\mu t}$ soit annulé par 
\begin{equation}\label{eq:34bis} 
-\partial_t^2 - \cH \partial_t + \lambda 
\end{equation}
sont $\mu_±=\frac{\cH}2±\sqrt{\frac{\cH^2}4+\lambda}$. 
A chaque valeur propre de $\tilde{A} + \cR_0$ est en réalité associée une paire de poids critiques~; 
celle correspondant au $i$-ème espace propre sera notée $( \mu_-^{(i)} , \mu_+^{(i)} )$.
L'expression \emph{poids critiques supérieurs} désignera la famille formée 
par tous les $\mu_+^{(i)}$ 
(dénomination analogue pour \emph{inférieurs})~; 
le poids critique supérieur (resp. inférieur) le plus petit (resp. grand) est 
$$\mu_+  =  \min_i \mu_+^{(i)} \quad \textrm{ (resp. } \mu_- = \max_i \mu_-^{(i)} \textrm{)}.$$ 
Le poids critique supérieur 
(resp. inférieur) le plus grand (resp. petit) sera dénoté $\mu_+^{max}$ (resp $\mu_-^{min}$), 
de telle sorte que tous les poids critiques sont situés dans l'intervalle $[\mu_-^{min},\mu_+^{max}]$. 

\smallskip

Il est montré dans \cite{Biq00} que, si l'opérateur $L$ est inversible dans $L^2$, alors :
\begin{itemize}
\item $L$ est un isomorphisme $C^{k,\alpha}_\delta\to C^{k-2,\alpha}_\delta$ pourvu que
  $\mu_-<\delta<\mu_+$ (où $C^{k,\alpha}_\delta=\frac 1{\cosh^\delta t}C^{k,\alpha}$) ;
\item il existe une constante $c>0$ telle que la fonction de Green $G$
  de $L$ soit contrôlée par la fonction de Green $G_\lambda$ de l'opérateur
  scalaire $\Delta+\lambda$ :
  \begin{equation}
    \label{eq:38}
    |G| \leq c G_\lambda .
  \end{equation}
\end{itemize}
Il en résulte qu'une solution $u$ de l'équation $Lu=v$ sur l'espace
hyperbolique satisfait $|u|\leq cu_0$, où $u_0$ est la fonction solution
de $(\Delta+\lambda)u_0=|v|$. Appliquant le lemme \ref{lemm:laplacien-scalaire},
on en déduit un contrôle de $u_0$ dans l'espace $C^0_{\delta_1,\delta_2}$, et
donc de $u$. Le contrôle des autres dérivées vient par régularité
elliptique, et il en résulte :

\begin{lemm}\label{lemm:P}
  Soit $L=\nabla^*\nabla+\cR_0$ inversible dans $L^2$, et $\lambda\geq 0$ la plus petite
  valeur propre de $\tilde{A}+\cR_0$. Si $\delta_1$ et $\delta_2$ satisfont les
  conditions du lemme \ref{lemm:laplacien-scalaire}, alors $L$ est un
  isomorphisme $C^{k,\alpha}_{\delta_1,\delta_2}\to C^{k-2,\alpha}_{\delta_1,\delta_2}$.\qed
\end{lemm}

\begin{rema}
  Le lemme reste vrai pour $\lambda<0$, mais comme nous n'avons pas traité
  ce cas (voir la remarque \ref{rema:lambda-negatif}), nous ne
  l'énonçons pas ici.
\end{rema}
\begin{rema}\label{rema:Kato}
  Le lemme reste vrai si on regarde le problème sur un demi-espace
  hyperbolique, avec condition de Dirichlet au bord. La démonstration
  en étant plus compliquée, nous nous contenterons de la remarque plus
  faible suivante : le lemme reste vrai dans ce cas, si $\lambda\geq 0$ est la
  plus petite valeur propre de $\cR_0$ (et non de $\tilde{A}+\cR_0$). En effet,
  on peut alors utiliser l'inégalité de Kato pour obtenir
  \begin{equation}
    \label{eq:40}
    \langle u,\nabla^*\nabla u\rangle = |u|\Delta|u|+|\nabla u|^2-|d|u| |^2 \geq |u|\Delta|u|,
  \end{equation}
  donc $(\Delta+\lambda)|u|\leq |\nabla^*\nabla u+\lambda u|$, ce qui permet d'appliquer 
le principe du maximum, et donc la démonstration du lemme
  \ref{lemm:laplacien-scalaire}.
\end{rema}

%
%
%
%
%
%
%

\section{Métriques asymptotiquement hyperboliques}
\label{sec:newtrois}

Nous nous plaçons ici sur une variété $M$ de dimension $n$ difféomorphe à
une demi-boule $B_+$ définie par
\begin{equation}
  \label{eq:43}
  B_+ =
  \begin{cases}
    \{x_1^2+\cdots+x_n^2< 1, x_1> 0\} &\text{ dans le cas réel,}\\
    \{\big(x_1+\frac{x_3^2+\cdots+x_n^2}{2}\big)^2+x_2^2<1, x_1> 0\} &\text{ dans
      le cas complexe.}  
  \end{cases}
\end{equation}
En effet, pour la métrique hyperbolique dans le modèle du demi-espace,
le domaine $B_+$ correspond exactement à un côté d'un hyperplan
totalement géodésique (dans le cas réel) ou d'un bisecteur (dans le
cas complexe). Nous distinguerons par la suite le bord à l'infini
\begin{equation*}
 \partial_{\infty}M \simeq \begin{cases}
    \{x_2^2+\cdots+x_n^2< 1, x_1= 0\} &\text{ dans le cas réel,}\\
    \{\big(\frac{x_3^2+\cdots+x_n^2}{2}\big)^2+x_2^2<1, x_1=0\} &\text{ dans
      le cas complexe}  
  \end{cases}
\end{equation*}
et le bord à l'intérieur 
\begin{equation*}
 \partial M =
  \begin{cases}
    \{x_1^2+\cdots+x_n^2= 1, x_1> 0\} &\text{ dans le cas réel,}\\
    \{\big(x_1+\frac{x_3^2+\cdots+x_n^2}{2}\big)^2+x_2^2 = 1, x_1> 0\} &\text{ dans
      le cas complexe,}  
  \end{cases}
\end{equation*}
tous deux difféomorphes à une boule ouverte d'un espace euclidien de
dimension $n-1$, et de bord commun la sphère
$$ S=\widebar{\partial M} \cap \widebar{\partial_\infty M} . $$
Observons maintenant que nous disposons de la dilatation, pour $t>0$,
\begin{equation}
  \label{eq:44}
  h_t(x_1,\dots,x_n) =
  \begin{cases}
    (tx_1,\dots,tx_n) &\text{ dans le cas réel,}\\
    (tx_1,tx_2,\sqrt t x_3,\dots,\sqrt t x_n) &\text{ dans le cas complexe.}
  \end{cases}
\end{equation}
Elle laisse invariante la métrique hyperbolique
\begin{equation}
  \label{eq:42}
  g_{h}=
  \begin{cases}
    \frac{dx_1^2+\cdots+dx_n^2}{x_1^2} &\text{ dans le cas réel,}\\
    \frac{dx_1^2+\eta^2}{x_1^2}+\frac{dx_3^2+\cdots+dx_n^2}{x_1} &\text{ dans
      le cas complexe.}
  \end{cases}
\end{equation}
Nous munissons maintenant $M$ d'une métrique asymptotiquement
hyperbolique réelle ou hyperbolique complexe. Cela
signifie que, près de $x_1=0$, la métrique $g$ est asymptote à
\begin{equation}
  \label{eq:35}
  g_0 = g_0(\gamma) :=
  \begin{cases}
    \frac{(dx_1)^2+\gamma}{x_1^2} &\text{ dans le cas réel,} \\
    \frac{(dx_1)^2+\eta^2}{x_1^2}+\frac \gamma x_1 &\text{ dans le cas complexe.}
  \end{cases}
\end{equation}
Dans le cas réel, $\gamma$ est une métrique sur $\partial_{\infty} M$ ; dans le cas
complexe, $\eta$ est une 1-forme de contact sur $\partial_{\infty} M$, on choisit une
structure presque complexe $J$ sur la distribution de contact $\ker
\eta$, de sorte que $\gamma(\cdot,\cdot)=d\eta(\cdot,J\cdot)$ soit une métrique sur $\ker \eta$~;
l'ensemble forme donc une structure (presque) CR (Cauchy-Riemann).
Comme notre but dans cet article est une étude locale près de
l'origine, on peut toujours se restreindre à une plus petite
demi-boule, ramenée à $B_+$ par dilatation. Aussi peut-on supposer dans le cas
complexe que dans les coordonnées locales

$(x_1,x_2,\dots,x_n)$ la forme de contact a la forme standard
(\ref{eq:6}) :
\begin{equation}
  \label{eq:41}
  \eta = dx_2 + \tfrac 12 (x_3 dx_4-x_4 dx_3 + x_5 dx_6-x_6 dx_5+\cdots ) .
\end{equation}
Enfin, on peut aussi supposer, quitte à modifier à nouveau les
coordonnées, qu'à l'origine, la métrique $\gamma$ est égale à celle du
modèle hyperbolique, à savoir $\gamma(0)=dx_2^2+\cdots+dx_n^2$ dans le cas réel,
et $\gamma(0)=dx_3^2+\cdots+dx_n^2$ dans le cas complexe. Dans ce cas, les
dilatations $h_t$ ramènent $g_0$ à $g_h$ quand $t$ tend vers $0$ :
\begin{equation}
  \label{eq:52}
  h_t^* g_0 \underset{t\to0}{\longrightarrow} g_h .
\end{equation}

On notera $\cM$ l'espace des métriques $\gamma$, toujours prises $C^\infty$. De
manière plus précise, nous dirons maintenant que $g$ est
asymptotiquement hyperbolique s'il existe $\delta>0$ tel que $g-g_0\in
C^{1,\alpha}_\delta$, où
\begin{equation}
  \label{eq:36}
  C^{k,\alpha}_\delta = x^\delta C^{k,\alpha}(g_0) 
\end{equation}
(où nous noterons désormais $x=x_1$).
Les métriques $g$, $g_0$, et la métrique hyperbolique $g_{h}$ sont
quasi-isométriques. Aussi les espaces de Hölder à poids (\ref{eq:36})
peuvent être définis indifféremment par rapport à $g_0$ ou $g_h$ (mais
pas par rapport à $g$ qui a priori n'est pas supposée lisse).

La donnée de la métrique $\gamma$ sur le bord n'est pas un invariant de la
métrique asymptotiquement hyperbolique. En effet, un changement de
coordonnées de la forme $\bar x = x e^{\omega}$ où $\omega$ est une fonction sur
$\partial_{\infty} M$ change $\gamma$ en $e^{2\omega}\gamma$ ou $e^\omega\gamma$ suivant le cas. La seule
donnée géométrique est donc la structure conforme $[\gamma]$ dans le cas
réel, ou la structure (presque) CR donnée par $[\eta]$ et $J$ dans le cas complexe
(par souci de simplicité, nous utiliserons le vocable «~classe
conforme~» dans les deux cas, la situation complexe pouvant être vue
comme une structure conforme sur une distribution de codimension~$1$).

Étant donnée une métrique asymptotiquement hyperbolique, il sera utile
de choisir un représentant de $[\gamma]$ adapté à la géométrie de la
demi-boule, et des coordonnées nous rapprochant des modèles
(\ref{eq:23}) et (\ref{eq:12}). Dans le cas réel, ce choix est fait
dans \cite[section 4]{Biq08} : sur $\partial_\infty M$ et partant d'une métrique $\gamma$
régulière jusqu'au bord $S$ inclus, on considère $y$ la distance
pour $\gamma$ au bord $S$, de sorte que près de $S$ on a 
\begin{equation}
\gamma=dy^2+\bar \gamma_y,\label{eq:53}
\end{equation}
où $\gamma_y$ est une famille de métriques sur $S$. On pose alors
\begin{equation}
  \label{eq:47}
  x = \frac{2u}{1+u^2}v , \quad y = \frac{1-u^2}{1+u^2} v ,
\end{equation}
pour obtenir
\begin{equation}
  \label{eq:48}
  g_0 = \frac 1{u^2}\Big( du^2 + \frac{(1+u^2)^2}4 \frac{dv^2
                         +\bar \gamma_y}{v^2} \Big) .
\end{equation}
Cette formule généralise l'écriture de la métrique hyperbolique comme
\begin{equation}
  \label{eq:51}
  g_h = \frac 1{u^2}\Big( du^2 + \frac{(1+u^2)^2}4 \frac{dv^2+dx_3^2+\cdots
    +dx_n^2}{v^2} \Big)
\end{equation}
(c'est la formule (\ref{eq:23}) avec $r=-\log u$).  Par rapport à la
coordonnée $u$, l'infini conforme de $g_0$ devient le représentant de
$[\gamma]$ donné par
\begin{equation}
  \label{eq:49}
  \tilde \gamma = \frac \gamma{4y^2} = \frac{dy^2+\bar \gamma_y}{4y^2},
\end{equation}
une métrique asymptotiquement hyperbolique sur $\partial_\infty M$.

\smallskip

Le cas complexe est plus délicat. Tous les points de la sphère de
Heisenberg $S$ ne sont pas équivalents, la structure de contact
devient tangente à $S$ aux deux points $p_±=(x_2=±1,x_3=\cdots x_n=0)$.
Nous commençons par rendre la structure CR standard en ces deux
points. Sans être absolument nécessaire, cette normalisation permet de
simplifier les calculs.
\begin{lemm}\label{lem:std-J}
  Quitte à faire agir un contactomorphisme de $\partial_\infty M$, on peut
  supposer que la structure CR $J$ coïncide avec la structure CR
  standard $J_0$ aux deux points $p_±$.
\end{lemm}
\begin{proof}
  C'est un fait classique, laissé au lecteur. Si $J$ est assez proche
  de $J_0$, on peut s'arranger pour que le contactomorphisme ne
  modifie qu'un petit voisinage de $p_+$ et $p_-$.
\end{proof}

Nous passons à une seconde normalisation, analogue au choix de la
coordonnée $y$ dans le cas réel de sorte que (\ref{eq:53}) soit
satisfaite. Prenons des coordonnées $Z=(z_i)$ sur la sphère de
Heisenberg épointée $S-\{p_+\}=\setC^{m-1}$, et une coordonnée transverse
$y$, de sorte que la structure de contact s'écrive sous la forme
standard $\eta=dy+\eta'$, où $\eta'=\frac 12\Im(\bar z_1dz_1+\cdots +\bar
z_{m-1}dz_{m-1}).$ Le point $p_-$ correspond donc à $z_i=0$.
Considérons sur $S$ la fonction 
$$ \varphi=|d_Hy|^2 , $$
où l'indice $H$ signifie qu'on évalue la norme de la restriction de
$dy$ à la distribution de contact. Pour le modèle hyperbolique
complexe, on calcule immédiatement $\varphi=\frac{r^2}4$ ; en général, la
fonction $\varphi$ ne s'annule pas sur $S-\{p_±\}$, et son comportement près
de $p_±$, par le lemme \ref{lem:std-J}, est celui du modèle, à savoir
$\varphi=\frac{r^2}4+ O((\frac r{1+r^2})^3)$ près de $p_±$.

\begin{lemm}\label{lem:std-phi}
  Quitte à faire agir un contactomorphisme de $\partial_\infty M$, préservant
  globalement $S$, on peut supposer que la fonction $\varphi$ est égale à sa
  valeur pour le modèle, à savoir $\varphi=\frac{r^2}4$.
\end{lemm}
\begin{proof}
  On fait agir l'exponentielle d'un champ de vecteurs préservant la
  structure de contact, $X_f=f \partial_y - \sharp d_Hf$, où $\sharp \alpha$ est le vecteur
  horizontal défini par $\sharp \alpha\lrcorner d\eta=\alpha$, et on choisit la fonction réelle
  $$ f(s,y) = g(s) y . $$
  Alors $X_f=y(g\partial_y-\sharp d_Hg) + \frac g2 r\partial_r$ est tangent au bord
  $S=\{y=0\}$, et y coïncide avec le vecteur radial $\frac g2 r\partial_r$. Il
  en résulte, le long de $S$,
  $$ \cL_{X_f}d_Hy = - \frac g2 r \partial_r \lrcorner d\eta = g d_Hy . $$
  En ajustant $g$, on peut s'arranger pour que le flot $(\psi_t)$ de
  contactomorphismes engendré par $X_f$ satisfasse $|\psi_1^*d_Hy| =
  \frac r2$, ou, de manière équivalente, $|d_Hy|_{(\psi_1)_*J}=\frac r2$.
\end{proof}

Maintenant, on écrit explicitement le changement de variables qui
permet d'obtenir les formules du lemme \ref{lemm:met-demi}. Étendons
la fonction $\varphi$ sur $B_+$ par la formule
$$ \varrho = \sqrt{(x+\varphi)^2+y^2} . $$
Posons alors
\begin{equation}
 u = \frac{2x}{x-\varphi+\varrho}, \quad v^2 = \frac{x-\varphi+\varrho}{x+\varphi+\varrho} .
 \label{eq:55}
\end{equation}
Utilisant
\begin{equation}
  \label{eq:63}
  V = v\sqrt{1-u},
\end{equation}
on en tire
\begin{equation}
 x = \frac{u v^2 \varrho}{1+V^2}, \quad y = \frac{2V\varrho}{1+V^2}, \quad
 \varphi = \varrho \frac{1-V^2}{1+V^2} .\label{eq:56}
\end{equation}
On observe facilement que $u\leq 1$ et $v\leq 1$. Le passage de $(x,y)$ à
$(u,v)$ présente une certaine analogie avec le cas réel.  Finalement,
on tordra les coordonnées en les ramenant par l'application $F$
définie par
\begin{equation}
  \label{eq:50}
  F(u,v,Z) = (x,y,e^{-i\arctan V}Z) .
\end{equation}

\begin{lemm}\label{lem:lien-hc}
  Les coordonnées dans lesquelles la métrique hyperbolique complexe a
  la forme du lemme \ref{lemm:met-demi} sont données par
$$ \frac 1u = \cosh^2\frac s2, \quad 
   \frac 1v = \cosh  \frac \rho2, \quad
   \varrho = e^{-2\tau}. $$
\end{lemm}
\begin{proof}
  Le calcul se fait à partir des formules (\ref{eq:14}) et (\ref{eq:15}).
\end{proof}

Passons maintenant à la métrique asymptotiquement hyperbolique
complexe $g_0=\frac{dx^2+\eta^2}{x^2}+\frac \gamma x$. On va maintenant
montrer qu'elle a une forme convenable dans les mêmes coordonnées que
pour $g_h$.

La métrique $\gamma$ le long de $S$ s'écrit
\begin{equation}
 \gamma = \frac{|d_Hy|^2}\varphi+\bar \gamma , \quad
 \bar \gamma=\frac{|Jd_Hy|^2+4\bar \gamma'}\varphi,\label{eq:61}
\end{equation}
où $4\bar \gamma'/\varphi$ est la métrique induite sur la distribution complexe
maximale $\cC\subset TS\cap \ker \eta$.

Pour la métrique hyperbolique $g_h$, on a
$\gamma=\frac{d\varphi^2}{\varphi}+4\varphi\gamma_{S^{2m-3}}$. En particulier, $\bar \gamma'=\gamma'$, où,
comme dans le lemme \ref{lemm:met-demi}, $\gamma'$ est la métrique standard
sur la distribution de contact de la sphère $S^{2m-3}$. Vu le lemme
\ref{lem:std-J}, dans le cas général, on a $\bar \gamma' \to \gamma' $ aux points
$p_±$. On obtient alors les formules suivantes, qui généralisent le
lemme \ref{lemm:met-demi} :
\begin{lemm}\label{lem:modele-local}
  Posons
  \begin{equation}
    \label{eq:58}
    \tilde \eta = \frac 1{v^2\varrho}\big(\eta'+\frac{2V}{1+V^2}d\varrho\big) .
  \end{equation}
  Alors
  \begin{equation}
    \label{eq:57}
    g_0 = \frac 1{1-u}\frac{du^2}{u^2}+\frac{\tilde \eta^2}{u^2}
         +\frac{\tilde \gamma}u+ O(\sqrt{u}v) ,
  \end{equation}
où $\tilde \gamma$ est une métrique sur le bord à l'infini $\{u=0\}$, définie
sur la distribution de contact, et satisfaisant, quand $v\to0$,
\begin{equation}
  \label{eq:60}
  \tilde \gamma = \frac{4 dv^2+\varrho^{-1}\bar \gamma}{v^2}+ O(v) .
\end{equation}
En outre, au voisinage de $p_±$ :
\begin{equation}
  \label{eq:59}
  \varrho^{-1}\bar \gamma \sim \frac{d\varrho^2}{\varrho^2} + 4 \gamma' .  
\end{equation}
Les dérivées covariantes des termes $O(\sqrt{u}v)$ sont aussi
$O(\sqrt{u}v)$.
\end{lemm}
\begin{rema}
  Il est clair dans cette formule (ou aussi bien dans celles du lemme
  \ref{lem:lien-hc}), qu'une équation lisse du bord intérieur $\{u=1\}$
  est $\sqrt{1-u}$ plutôt que $1-u$.
\end{rema}
\begin{proof}
Le premier terme donne :
\begin{align}
  \frac{dx^2}{x^2}
 &= \left( \frac{du}u+2\frac{dv}v+\frac{d\varrho}\varrho-\frac{2dV}{1+V^2}
    \right)^2 \notag\\
 &= \frac{du^2}{u^2}+4\frac{du}u\frac{dv}v+4\frac{dv^2}{v^2}
    + O(\sqrt u v) .\label{eq:64}
\end{align}
Le $O(\sqrt u v)$ est évalué par rapport à la forme (\ref{eq:57}) que
nous visons pour la métrique.

Un calcul plus délicat, prenant en compte la rotation (\ref{eq:50})
sur la coordonnée $Z$, est laissé au lecteur, qui arrivera à l'égalité
sur les formes de contact :
$$ \frac{\eta}{x} = \frac{\tilde \eta}{u} . $$

Il reste à examiner le terme $\frac \gamma x=\frac {(1+V^2)\gamma}{uv^2\varrho}$. Remarquons
que $d_Hy=dy-\eta$, d'où résulte $\frac{|d_Hy|^2}x=\frac{|dy|^2}x+
O(\sqrt x)=\frac{|dy|^2}x+ O(\sqrt u v \varrho)$. Ainsi dans le calcul, peut-on
remplacer $d_Hy$ par $dy$ sans modifier l'asymptotique de la
métrique. Or, après calcul,
\begin{multline*}
 \frac{|dy|^2}{\varphi x} = \frac 4{u(1-V^2)} \Bigg(
   (1-u) \Big( \frac{1-V^2}{1+V^2} \frac{dv}v + \frac{d\varrho}\varrho \Big)^2
 \\ + \frac{1-V^2}{1+V^2} \frac{du^2}{4(1-u)}
 - \frac{1-V^2}{1+V^2} \Big( \frac{1-V^2}{1+V^2} \frac{dv}v +
            \frac{d\varrho}\varrho \Big) \Bigg).
\end{multline*}
Analysant chaque terme, il reste
\begin{equation}
  \label{eq:65}
  \frac{|dy|^2}{\varphi x}
 = \frac {4(1-u)}u \Big(\frac{dv^2}{v^2} + O(v)\Big) 
  - \frac{du^2}{1-u} 
  - 4 \frac{du}u \frac{dv}v
  + O(\sqrt u v).
\end{equation}
L'addition de (\ref{eq:64}) et (\ref{eq:65}) donne
$$
\frac{dx^2}{x^2}+\frac{|dy|^2}{\varphi x}
= \frac 1{1-u} \frac{du^2}{u^2}
 + \frac 1u\Big(\frac{dv^2}{v^2} + O(v)\Big)
 + O(\sqrt u v).
$$
Le reste de la métrique se traite de manière similaire.
\end{proof}

%
%
%
%
%
%
%
%
%

\section{Jauge locale}

\label{sec:trois}

Nous continuons ici notre étude des métriques asymptotiquement
hyperboliques.  Nous disposons sur l'espace modèle des coordonnées
adaptées à un demi-espace définies dans les sections précédentes et
des espaces à double poids $C^{k,\alpha}_{\delta_1,\delta_2}$ de la section
\ref{sec:deux}. Par rapport à cet espace à poids, la condition
initiale $g-g_0\in C^{1,\alpha}_\delta$ implique $g-g_0\in C^{1,\alpha}_{\delta,\delta}$. Plus
généralement, nous avons

\begin{lemm}\label{lemm:inclusion}
  Pour tous $k\in\NM$, $\alpha\in]0,1[$, et tout $\delta\geq 0$, l'espace $C^{k,\alpha}_\delta$
  est inclus dans $C^{k,\alpha}_{\delta,\delta}$. La même conclusion est valable sur
  un demi-espace hyperbolique.
\end{lemm}

\begin{proof}
  C'est une conséquence immédiate des changements de variable calculés
  dans la section précédente. D'une part, dans le cas réel, on a
  $$x=\frac{2uv}{1+u^2} \quad \textrm{ avec } w=u^{-\delta_1}v^{-\delta_2}.$$ 
  D'autre part, dans le cas complexe, on a 
  $$x=u v^2 \tilde \psi \quad \textrm{ avec } w=u^{-\delta_1}(v^2\psi)^{-\delta_2},$$
  ce qui conclut la preuve.
\end{proof}

La première étape de la démonstration de la régularité locale consiste à
trouver une jauge locale pour l'action des difféomorphismes pour la
mé­tri­que $h_t^*g$, si $t$ est assez petit (c'est-à-dire en se
restreignant à une plus petite boule si nécessaire). Soit l'opérateur
de Bianchi agissant sur les formes bilinéaires symétriques,
\begin{equation}
  \label{eq:46}
  B_g = \delta_g + \tfrac 12 d \Tr_g .
\end{equation}
Définissons un groupe de difféomorphismes $\cD^{k,\alpha}_{\delta_1,\delta_2}$ de la
demi-boule $M$, induisant l'identité au bord à l'infini $x_1=0$, en
disant que $\Phi\in\cD^{k,\alpha}_{\delta_1,\delta_2}$ si $\Phi=\exp_{g_0}(X)$ pour un champ
de vecteurs $X\in C^{k,\alpha}_{\delta_1,\delta_2}$. Soit $g_0$ est la métrique définie
par \eqref{eq:35}, et $\bar g$ une autre métrique, fixée, avec le même
comportement asymptotique : $\bar g-g_0\in C^{\infty}_{\delta,\delta}$. Alors on a :

\begin{lemm}\label{lemm:jauge}
  Pour tout $t>0$ assez petit, il existe un 
  difféomorphisme $\Phi\in \cD^{2,\alpha}_{\delta,\delta}$ de $M$, induisant l'identité 
  au bord intérieur $\partial M$, tel que
  $$ B_{h_t^*\bar g}(\Phi^*h_t^*g) = 0. $$
\end{lemm}

\begin{proof}
  Observons que quand $t\to0$, alors $h_t^*\gamma\to\gamma_{h}$ et $h_t^*g\to g_{h} $,
  avec plus précisément
  \begin{equation}
  \label{eq:45}
  \| h_t^*(g-\bar g) \|_{C^{1,\alpha}_\delta} = O(t^\delta) ,\quad  
 \| h_t^*\bar g-g_h \|_{C^{k,\alpha}_0} = O(t).
  \end{equation}
  La première inégalité est immédiate, la seconde est une conséquence
  des raisonnements tenus dans le début de la preuve de la Proposition
  I.3.5 de \cite{Biq00}.

  On écrit alors l'équation à résoudre comme
  $B_{\Phi_*h_t^*\bar g}(h_t^*g)=0$.  On obtient le résultat en
  appliquant le théorème des fonctions implicites à l'opérateur
  $$ \mathcal{B}:\cD^{2,\alpha}_{\delta,\delta} × \cM × C^{\infty}_{\delta,\delta} × C^{1,\alpha}_{\delta,\delta} 
  \longrightarrow C^\alpha_{\delta,\delta} , $$ défini par
  $\mathcal{B}(\Phi,\gamma,h_0,h)=B_{\Phi_*\bar g}(h)$. En effet, d'après
  \cite{Biq00}, la différentielle $\frac{\partial\mathcal{B}}{\partial\Phi}$ au point
  $(\textrm{id},\gamma_{h},0,0)$ est
  $B_{g_h}\delta^*_{g_h}=\frac12(\nabla^*\nabla-\Ric(g_h))$ où $\nabla$ est la connexion
  de Levi-Civita de $g_h$. C'est un isomorphisme par le lemme
  \ref{lemm:P}, y compris lorsqu'il est restreint aux champs de
  vecteurs nuls au bord intérieur par la remarque \ref{rema:Kato}.
\end{proof}

\begin{rema}
  On n'a pas vraiment utilisé dans cette démonstration les espaces
  fonctionnels à double poids $C^{k,\alpha}_{\delta_1,\delta_2}$, puisque la
  résolution pouvait aussi bien se faire dans $C^{k,\alpha}_\delta$. Leur
  utilité apparaîtra plus loin, quand on augmentera $\delta_1$ en laissant
  fixe $\delta_2$. Cette méthode permettra d'ignorer les problèmes de
  comportement des solutions au coin formé par le bord commun au bord intérieur et au
  bord à l'infini.
\end{rema}
 
Le lemme suivant montre que si $g$ satisfait en outre l'équation
d'Einstein, elle est nécessairement lisse dans la jauge précédente :

\begin{lemm}\label{lem:reg-locale}
  Supposons que $g-g_0\in C^{1,\alpha}_{\delta,\delta}$ et que $g$ soit Einstein. Alors,
  dans la jauge construite par le lemme \ref{lemm:jauge}, pour $t$
  assez petit, on a $h_t^*(g-\bar g)\in C^\infty_{\delta,\delta}$, avec pour tout $k$ une
  estimation $\|h_t^*(g-\bar g)\|_{C^k_{\delta,\delta}}\leq c_k \|h_t^*(g-\bar g)\|_{C^1_{\delta,\delta}}$.
\end{lemm}
\begin{proof}
  On se place pour $t>0$ assez petit, après avoir appliqué la
  dilatation $h_t$, que l'on omettra dans la suite de la
  démonstration. La question est locale, et on est donc ramené à la
  traiter dans une boule pour $\bar g$, avec $\|g-\bar
  g\|_{C^{1,\alpha}_{\delta,\delta}}$ petite (donc en particulier $\|g-\bar g\|_{C^0}$
  est petite). Dans une boule pour $\bar g$ la variation des poids est
  bornée et l'estimation nécessaire se ramène à estimer $\|g-\bar
  g\|_{C^{k,\alpha}}$ par $\|g-\bar g\|_{C^{1,\alpha}}$.  C'est un énoncé
  entièrement analogue à la régularité des métriques d'Einstein dans
  des coordonnées harmoniques, et on sera donc concis.  La condition
  de jauge $B_{\bar g}g=0$ s'écrit
$$ {\bar g}^{il}(\partial_lg_{ij}-\frac12 \partial_jg_{il})
  = \text{termes d'ordre $0$ en $g$} .
$$
  D'un autre côté,
$$ \Ric^g_{jk} = \frac12 g^{il}(-\partial^2_{il}g_{jk}+\partial^2_{ik}g_{lj}+\partial^2_{ij}g_{lk}-\partial^2_{jk}g_{il})
  + \text{termes d'ordre $\leq 1$ en $g$} .
$$
  L'équation d'Einstein implique donc, en différenciant la première
  équation,
$$ - {\bar g}^{il} \partial^2_{il}g_{jk}  + (g^{il}-\bar g^{il})(\partial^2_{ik}g_{lj}+\partial^2_{ij}g_{lk}-\partial^2_{jk}g_{il})
  = \text{termes d'ordre $\leq 1$ en $g$}.
$$
  C'est une équation de la forme $$-h_{jk}^{abcd}\partial^2_{ab}g_{cd}=\text{termes
    d'ordre $\leq 1$ en $g$},$$ avec 
$$h_{jk}^{abcd}=\bar g^{ab}\delta_{jk}^{cd}+\text{termes d'ordre $0$ en $g-\bar g$}.$$
Ce qui est important ici est que $\|g-\bar g\|_{C^0}$ est petit, donc
l'équation est elliptique. Il est alors standard d'en déduire la
régularité de $g$.
\end{proof}

%
%
%
%
%
%
%
%
%

\section{Analyse sur les variétés asymptotiquement hyperboliques}
\label{sec:quatre}

Dans cette section, nous préparons l'analyse de la polyhomogénéité en
introduisant des espaces de sections polyhomogènes et en analysant
l'opérateur indiciel.

Soit $g$ une métrique asymptotiquement hyperbolique définie sur une
variété $M$  difféomorphe à une demi-boule. Les dilatations $h_t$ de la section précédente 
fournissent une métrique hyperbolique de référence $g_h$ dans $M$, induite par
$\gamma_h = \lim_{t\to 0} h_t^* \gamma$. Nous disposons alors des métriques $g_h$, $g_0$
et $g$ sur la demi-boule. Pour éviter de considérer l'analyse sur la demi-boule $B_+$ (et donc d'avoir à
introduire une condition au bord sur le bord intérieur), nous
recollerons $g$ (ou $g_0$) avec la métrique hyperbolique sur la demi-boule
opposée.  Plus précisément, l'espace hyperbolique est la réunion $B_+\cup
B_-$, avec $B_+=\{s\geq 0\}$ et $B_-=\{s\leq 0\}$. Nous définissons alors une
famille de métriques $g_t$ sur $\setR H^n$ ou $\setC H^m$ par
\begin{equation}\label{eq:69}
 g_t = \chi(s) h_t^* g + (1-\chi(s)) g_h ,
\end{equation}
où $\chi$ est une fonction de coupure satisfaisant $\chi(s)=1$ pour $s\geq 1$
et $\chi(s)=0$ pour $s\leq 0$ (la même construction peut être appliquée à la métrique $g_0$, ou encore 
à une métrique $\bar g$ telle que dans la section précédente,
conduisant à une famille notée ci-dessous $\bar g_t$). Il est
important de noter que les métriques obtenues sont à courbure
sectionnelle strictement négative si ces recollements sont faits suffisamment près de l'infini, autrement dit
si $t$ est suffisamment petit.

Nous pouvons alors choisir
une demi-boule incluse dans la demi-boule d'origine et y effectuer la construction d'un
représentant asymptotiquement hyperbolique dans la classe conforme du bord, comme à la fin de la
section \ref{sec:newtrois}. C'est 
cette mé­trique que nous noterons désormais $\gamma$.

Nous utiliserons à partir de maintenant, ici comme dans
toute la suite, des coordonnées de type exponentiel, comme dans les
expressions \eqref{eq:24} et \eqref{eq:12} des sections \ref{sec:un}
et \ref{sec:deux}, plus adaptées pour la suite que celles de
l'expression \eqref{eq:35} précédente. Ainsi $M$ est-il identifié à un
produit $[0,+\infty)×\partial_\infty M$, et le terme «~dominant~»
$g_0$ des $g_t$ est de la forme
\begin{equation}
  \label{eq:35again}
 g_0 = g_0(\gamma) = \begin{cases}
    ds^2+e^{2s}\gamma &\text{ dans le cas réel,} \\
    ds^2+e^{2s}\eta^2 + e^{s}\gamma &\text{ dans le cas complexe,}
  \end{cases}
\end{equation}
avec $\gamma = d\eta(\cdot,J\cdot)$ dans le second cas. Rappelons que $\gamma$ est ici la
métrique sur $\partial_\infty M$, au sein de la classe conforme à l'infini, adaptée
à la géométrie de la demi-boule.

\smallskip

Les espaces fonctionnels utiles $C^{k,\alpha}_{\delta_1,\delta_2}$ sont les mêmes que
précédemment, les dérivées étant celles de l'espace hyperbolique
modèle ou celles de la connexion $\nabw$ définie plus bas, et les
normes étant prises relativement à $g$, $g_0$ ou $g_h$, ce qui revient
au même de par le lemme \ref{lemm:connexions} énoncé plus loin, et du
fait que ces trois métriques sont mutuellement bornées.

\smallskip

Nous notons $\dA_{\alpha,\beta}$ l'espace des fonctions polyhomogènes finies 
\emph{de poids au moins} $\alpha\geq 0$ dans la direction radiale, c'est-à-dire des sommes finies 
du type $\sum_{\sigma,\tau} a_{\sigma,\tau}  \, s^{\sigma} e^{-\tau s}$, où 
\begin{itemize}
\item les $a_{\sigma,\tau}$ sont des sections \emph{lisses sur le bord à l'infini} (autrement dit 
constantes en $s$ dans la trivialisation $[s_0,+\infty[× \partial M$) et vivant le long des tranches
dans un espace à simple poids, dont la valeur est fixée égale à $\beta>0$~;
\item les $\sigma$ et $\tau$ sont des réels~;
\item tous les poids $\tau$ rencontrés dans la somme sont supérieurs ou égaux à $\alpha$.
\end{itemize} 
Cette définition s'étend immédiatement à des sections de fibrés naturels à condition de prendre les coefficients 
$a_{\sigma,\tau}$ \emph{parallèles} selon $\partial_s$ (en raison de la croissance exponentielle de la métrique,
la trivialisation par les sections radialement parallèles diffère de la trivalisation en coordonnées d'une puissance
de $e^s$).

Nous introduisons également les sous-espaces $\cA_{\alpha,\beta}$
formé des éléments sans puissances polynômiales (c'est-à-dire 
que l'on a nécessairement $\sigma=0$) et $\dA[\alpha]_{\beta}$ formé des sommes 
$$ \sum_{\sigma} a_{\tau}  \, s^{\sigma} e^{-\tau s}  $$
dont le comportement en norme est exactement égal à $e^{-\alpha s}$ (formellement, il s'agit donc du quotient
$\dA_{\alpha,\beta}/\cup_{\alpha'>\alpha}\dA_{\alpha',\beta}$ mais nous nous passerons de cette subtilité ici). 
Si $u\in\dA_{\alpha,\beta}$, on notera $[u]_{\alpha}$ sa composante de poids $\alpha$, c'est-à-dire dans 
$\dA[\alpha]_{\beta}$.

Enfin, nous dirons qu'une fonction (ou une section) $u$ est \emph{polyhomogène}
s'il existe une série formelle  $\sum_{\sigma,\tau} a_{\sigma,\tau}  \, s^{\sigma} e^{-\tau s}$ telle que
pour tout $\delta>0$, il existe une somme finie  $\sum_{|\tau|\leq N} a_{\sigma,\tau}  \, s^{\sigma} e^{-\tau s}$
telle que 
$$ u - \sum_{|\tau|\leq N} a_{\sigma,\tau}  \, s^{\sigma} e^{-\tau s} \ \in \ C^{\infty}_{\delta,\beta}.$$
De fa\c{c}on équivalente, nous dirons que $u$ appartient à $\dA_{\alpha,\beta}$ \emph{modulo} $C^{\infty}_{\infty,\beta}$. 
Cette définition dépend de la valeur de $\beta$, mais dans la suite $\beta$
sera fixé et il n'y aura donc pas d'ambiguïté sur la définition. 

\smallskip

Nous pouvons maintenant passer à l'analyse dans le cadre asymptotiquement hyperbolique. 
Notre point de départ est une étude précise du terme dominant de la connexion de Levi-Civita de la métrique $g_0$.
 
\begin{defi} On peut étendre la connexion de Levi-Civita de $\gamma$ dans
  le cas réel, resp. de Tanaka-Webster dans le cas complexe,
  en une connexion $\nabw$ à l'intérieur de $M$,
  unitaire pour $g_0$, définie dans le cas réel par
$$ \nabw\partial_s =0, \ \nabw_{\partial_s}(e^{-s}h) = 0 $$
pour tout champ de vecteur $\gamma$-unitaire $h$ sur le bord, et dans le cas complexe par
$$ \nabw\partial_s =0, \ \nabw_{\partial_s}(e^{-s}R) = 0, \ \nabw_{\partial_s}(e^{-\tfrac{s}{2}}h) = 0$$
pour tout champ de vecteur $\gamma$-unitaire $h$ dans le noyau de $\eta$ sur le bord et où $R$ désigne 
le champ de Reeb de la forme de contact $\eta$.
\end{defi}

Les faits suivants sont démontrés dans \cite{BiqHer05} en dimension
$4$ et dans le cas complexe~; il s'obtiennent aisément dans le cas
général en suivant les mêmes idées (ils sont d'ailleurs immédiats dans
le cas réel).

\begin{lemm}\label{lemm:connexions}
  Dans une base $g_0$-orthonormée adaptée $\{\partial_s, e^{-s}h_i\}$ dans le
  cas réel, resp.  $\{\partial_s, e^{-s}R, e^{-s/2}h_i, e^{-s/2}Jh_i\}$ dans le
  cas complexe, la connexion de Levi-Civita de $g_0$ 
  diffère de $\nabw$ par une $1$-forme à valeurs dans les
  endomorphismes de $TM$ de la forme
$$ a = a_0 + a_1 , $$
où $a_0$ est à coefficients \emph{constants} (par rapport à $\nabw$)
et $a_1$ est un élément de $\cA_{1,1}$ dans le cas réel, $\cA_{1/2,1}$
dans le cas complexe.  De même, la courbure est
$$ R = R_0 + R_1,$$ 
où $R_0$ est à coefficients constants, égale à la courbure de l'espace hyperbolique modèle, et $R_1$ est un élément 
de $\cA_{1,1}$ dans le cas réel, $\cA_{1/2,1}$ dans le cas complexe.
\end{lemm}

\smallskip

Nous considérons maintenant un opérateur différentiel $L = \nabla^*\nabla  + \mathcal{R}$
où $\mathcal{R}$ est un terme de courbure, agissant sur les sections d'un fibré tensoriel naturel $E$.
Bien sûr, nous avons en vue le cas où $L$ est la différentielle en une métrique
$\bar g$ de l'opérateur
$$ g \longmapsto \Ric^{g} + c g + (\delta^{g})^*B_{\bar g}(g). $$
Un premier résultat, élémentaire mais important dans ce contexte, est le suivant : 

\begin{lemm}\label{lemm:Pbis}
  Supposons que, pour la métrique hyperbolique, l'opérateur $L_h$ soit
  inversible dans $L^2$. Alors, pour $t>0$ assez petit, l'opérateur
  $L_t$ pour la métrique $\bar g_t$ est un isomorphisme
  $C^{k,\alpha}_{\delta_1,\delta_2}\to C^{k-2,\alpha}_{\delta_1,\delta_2}$, pour les poids $(\delta_1,\delta_2)$
  satisfaisant les hypothèses du lemme \ref{lemm:P}.
\end{lemm}

\begin{proof}
  Par la seconde estimation de \eqref{eq:45}, on a
  $\|\bar g_t-g_h\|_{C^{k,\alpha}}=O(t)$. Il en résulte que l'opérateur
  $L-L_h$ est petit en norme d'opérateur $C^{k,\alpha}_{\delta_1,\delta_2}\to
  C^{k-2,\alpha}_{\delta_1,\delta_2}$. Puisque les hypothèses assurent
  l'inversibilité de $L_h$ entre ces mêmes espaces, l'opérateur $L$
  lui aussi est inversible pour $t>0$ assez petit.
\end{proof}

Nous noterons désormais $g$, $\bar g$ ou $g_0$ les métriques construites avec un choix de $t$ suffisamment 
petit pour que le résultat du lemme précédent soit valide pour
l'opérateur $L$ considéré.

\smallskip

Il nous faut maintenant analyser plus en détails l'opérateur $L_{g_0}$
bâti à partir de la métrique $g_0$. Dans la décomposition
$M=[0,+\infty[×\partial_\infty M$, on peut trivialiser $E$ le long de chaque rayon
$\setR_+×\{y\}$ en y choisissant une base parallèle pour la connexion
$\nabw$. Dans cette trivialisation, la dérivée $\tilde \nabla_{\partial_s}$
devient la dérivation ordinaire $\partial_s$ (qu'on prendra garde de ne pas
confondre avec la dérivée de Lie $\cL_{\partial_s}$).

La décomposition $\nabla^{g_0}=\nabw + a_0 +a_1$ implique alors que l'opérateur $L_{g_0}$ est asymptote sur 
chaque sous fibré homogène à un opérateur de la forme
$$ -\partial_s^2 - \cH \partial_s + \tilde{A} + \cR_0 + \Ltrans ,$$
où ici tous les termes sont indépendants de $s$ : $\tilde{A}$ rassemble tous
les termes constants d'ordre $0$ dans le laplacien (ce sont les mêmes
que ceux de l'opérateur de l'espace hyperbolique), de même $\cR_0$
contient les termes constants de la courbure (à nouveau, il s'agit
simplement des termes de courbure de l'espace hyperbolique), enfin
$\Ltrans$ contient toutes les dérivées transverses à $\partial_s$
(typiquement, par exemple dans le cas complexe, des termes comme
$-(\nabw_{e^{-s/2}h})^2$).
L'opérateur indiciel de $L$ est alors
$$ -\partial_s^2 - \cH \partial_s + \tilde{A} + \cR_0 . $$

\smallskip

Deux opérateurs intégraux, définis pour un couple de réels $(\alpha_+,\alpha_-)$ avec 
$\alpha_+>\alpha_-$, jouent un rôle important dans ce qui suit. Ils sont
définis sur les fonctions par~:
$$  G_{\infty}(u) (s) = 
\frac{1}{\alpha_+-\alpha_-}
\left( e^{-\alpha_- s}\int_{+\infty}^s e^{\alpha_-\varsigma}\, u(\varsigma)\, d\varsigma \ - \ 
 e^{-\alpha_+ s}\int_{+\infty}^s e^{\alpha_+\varsigma}\, u(\varsigma) \, d\varsigma \right) $$
et de
$$  G_{0}(u) (s) = 
\frac{1}{\alpha_+-\alpha_-}
\left( e^{-\alpha_- s}\int_{+\infty}^s e^{\alpha_-\varsigma}\, u(\varsigma) \, d\varsigma \ - \ 
 e^{-\alpha_+ s}\int_{s_0}^s e^{\alpha_+\varsigma}\, u(\varsigma) \, d\varsigma \right), $$
où $s_0$ est un réel positif supposé très grand par la suite. Ils possèdent les propriétés suivantes~:
\begin{enumerate}
\item\label{e1} Si $\alpha>\alpha_+$, l'opérateur $G_{\infty}$ envoie les espaces 
$C^{\infty}_{\alpha,\beta}$, 
$\cA_{\alpha,\beta}$, $\dA_{\alpha,\beta}$ et $\dA[\alpha]_{\beta}$ dans eux-mêmes~;
\item\label{e2} Si $\alpha_-<\alpha\leq \alpha_+$, l'opérateur $G_0$ envoie 
$\dA_{\alpha,\beta}$ dans lui-même, et $C^{\infty}_{\alpha,\beta}$ et 
$\mathcal{A}_{\alpha,\beta}$ dans eux-mêmes seulement si
$\alpha<\alpha_+$ ; en revanche, il y a apparition d'un terme polynômial si $\alpha=\alpha_+$. 
\end{enumerate}
Ces opérateurs sont des inverses \emph{à droite} de
$\partial_s^2 + \mathcal{H}\partial_s - \lambda$ si le couple $(\alpha_-,\alpha_+)$ 
est relié à la constante $\lambda$ par 
$$ \alpha_{±} = \tfrac{\mathcal{H}}{2} ± \sqrt{\tfrac{\mathcal{H}^2}{4} + \lambda}\ .$$

Leur définition s'étend aux sections du fibré tensoriel $E$, en le
trivialisant comme ci-dessus radialement sur $[0,+\infty[×\partial M$. Alors
$\tilde{A}+\cR_0$ est constant pour $\nabw$ le long des rayons, et ses
valeurs propres décomposent $E=\oplus_\lambda E_\lambda$ en la somme de ses
sous-espaces propres.

\smallskip

Notre boîte à outils est complétée par une algèbre d'opérateurs adaptée,
définie dans le cas réel comme $\cQ = \cQ_0 + e^{-s}\cQ_1$, 
où 
\begin{itemize}
\item $\cQ_1$ est l'algèbre engendrée par $\nabw_{\partial_s}$, les $e^{-s}\nabw_H$ 
(la notation $H$ en indice désignant n'importe quelle dérivée dans une direction tangente au bord à l'infini et
indépendante de $s$) et $\dA_{0,0}$~;
\item $\cQ_0$ est le sous-espace vectoriel de $\cQ_1$ ne contenant que des dérivées 
transverses à $\partial_s$ et des fonctions de $\dA_{1,0}$~;
\end{itemize}
et dans le cas complexe comme $\cQ = \cQ_0 + e^{-s/2}\cQ_1$, où 
\begin{itemize}
\item $\cQ_1$ est l'algèbre engendrée par $\nabw_{\partial_s}$, $e^{-s}\nabw_R$, 
les $e^{-s/2}\nabw_h$ ($h$ désignant ici une direction dans $H=\Ker\eta$ toujours indépendante de
$s$) et $\dA_{0,0}$~;
\item $\cQ_0$ est le sous-espace vectoriel de $\cQ_1$ ne contenant que des dérivées 
transverses à $\partial_s$ et des fonctions de $\dA_{1/2,0}$.
\end{itemize}
L'utilité de ces algèbres réside dans deux lemmes :
\begin{lemm}\label{lem:mapQ}
  Soit un opérateur $P\in \cQ$. Alors $P$ envoie $C^{k,\alpha}_{\delta_1,\delta_2}$
  dans $C^{k,\alpha}_{\delta_1-1,\delta_2}$ dans le cas réel,
  resp. $C^{k,\alpha}_{\delta_1-1/2,\delta_2}$ dans le cas complexe.\qed
\end{lemm}

\begin{lemm}
Dans le cas réel, $[\nabw_{\partial_s},\nabw_H]=0$ et $e^{-s}[\nabw_{H},\nabw_{H}]\in\cQ$~; en 
conséquence $[L_{g_0},\nabw_H] \in \cQ$. 

Dans le cas complexe, $[\nabw_{\partial_s},\nabw_H]=[\nabw_{\partial_s},\nabw_R] =0$,
$[\nabw_{e^{-s}R},\nabw_H]\in \cQ$, et pour des champs de vecteurs $h_1$ et $h_2$ $\gamma$-unitaires 
dans $H$, 
$$
[\nabw_{e^{-s/2}h_1},\nabw_{h_2}]= -d\eta(h_1,h_2)\,\nabw_{e^{-s/2}R} \ \textrm{ mod. } \ \cQ.
$$
En conséquence, $[L_{g_0},\nabw_R] \in \cQ$, et 
$[L_{g_0},\nabw_h] = -2\nabw_{e^{-s/2}Jh}\nabw_{e^{-s/2}R}$ mod. $\cQ$.
\end{lemm}

\begin{proof} 
Facile dans le cas réel, voir \cite{BiqHer05} dans le cas complexe en dimension $4$, dont la preuve s'étend presque mot
pour mot au cas général. 
\end{proof}

Le lemme reste vrai si $L_{g_0}$ est remplacé par tout autre opérateur dans ${\mathcal Q}$ et 
sera appliqué plus loin à $L_g$ pour les métriques $g$ telles que $g-g_0$ soit polyhomogène (donc $L_g-L_{g_0}$ sera
essentiellement dans $\cQ$).

%
%
%
%
%
%
%
%
%

\section{Décroissance des dérivées transverses}
\label{sec:newcinq}

Comme précédemment, nous considérons ici une métrique asymptotiquement
hyperbolique $g$ définie sur une demi-boule.

\begin{conv} 
  Toutes les fonctions et sections qui apparaissent dans la suite,
  définies \emph{a priori} sur un voisinage du bord à l'infini, seront
  étendues à l'es-pace tout entier par le procédé de troncature
  suivant~: choisissant une fois pour toutes une fonction lisse $\chi$
  d'une variable telle que $\chi(s)=0$ si $s\leq s_0/2$ et $\chi(s)=1$ si $s\geq
  s_0$ ($s_0$ suffisamment grand), si $u$ est une section définie sur
  la demi-boule et appartenant à l'espace fonctionnel à double poids
  $C^{k,\alpha}_{\delta_1,\delta_2}$ de cette demi-boule, $\varphi u$ est définie sur
  l'espace tout entier et appartient à l'espace fonctionnel de mêmes
  poids défini sur l'espace tout entier. Ce dernier fait est une
  conséquence du lemme \ref{lemm:uniform}.

Cette convention sera en vigueur dans toute cette section comme dans la suivante, 
et nous ne distinguerons pas dans les notations
les sections initialement définies au voisinage de l'infini de leurs versions tronquées et étendues.
\end{conv}

Nous considérons toujours un opérateur (issu de la métrique $g$) du type $L=\nabla^*\nabla +
\mathcal{R}$ agissant sur les sections d'un fibré tensoriel naturel
$E$, et nous supposerons que l'opérateur hyperbolique modèle $L_h$ est
inversible. Après recollement de la métrique $g$ avec la métrique
hyperbolique sur l'autre demi-boule (voir le début de la section \ref{sec:quatre}) et 
choix d'une valeur de $t$ suffisamment petite,
le Lemme \ref{lemm:Pbis} implique immédiatement le résultat suivant, que nous appellerons \emph{lemme de 
régularité}~:

\begin{lemm} \label{lemm:regularite}
Soit $\ell\geq k \geq 2$, $\alpha\in]0,1[$. Si une section $u$ 
vérifie $$  u \in C^{k,\alpha}_{\delta',\beta} \ \textrm{ et } \ L(u) \in C^{\ell-2,\alpha}_{\delta,\beta},$$
où $\mu_-\leq 0< \delta' \leq \delta <\mu_+$ et $(\delta,\beta)$ vérifient les conditions du lemme \ref{lemm:P}, 
alors nécessairement $u \in C^{\ell,\alpha}_{\delta,\beta}$.
\end{lemm}

L'objectif de ce paragraphe est de démontrer que le lemme de régularité entraîne un résultat 
de décroissance des dérivées transverses. Afin d'énoncer plus facilement le résultat, nous appellerons 
{\it opérateur polyhomogène au moins quadratique} (d'ordre $2$) un opérateur $q$ tel que
\begin{enumerate}
\item il existe
une série formelle 
$$\sum_{i\geq 2} P_i$$ 
où chaque $P_i$ est un polynôme de degré $i$ en  
$$Id, \ \partial_s , \ e^{-s}\nabw_H, \ \partial_s^2, \ e^{-2s}(\nabw_H)^2, \ e^{-s}\nabw_H\partial_s $$
dans le cas réel, et en
$$
Id, \ \partial_s , \ e^{-s/2}\nabw_H, \ e^{-s}\nabw_R, \ \partial_s^2, \ e^{-2s}(\nabw_R)^2,
\ e^{-s}(\nabw_H)^2,   $$
$$ 
e^{-s}\nabw_R\partial_s, \ e^{-s/2}\nabw_H\partial_s, \ e^{-3s/2}\nabw_R\nabw_H
$$ 
dans le cas complexe, à coefficients polyhomogènes dans $\dA_{0,0}$~;
\item pour tout $\delta>0$, il existe $N\geq 2$ tel que pour tout $u$ dans $C^{k+2,\alpha}_{\delta_1,\delta_2}$,
$$q(u) -  \sum_{2\leq i\leq N} P_i(u)  \ \in \ C^{k,\alpha}_{\delta_1+\delta,\delta_2}.$$
\end{enumerate}
Un exemple typique d'un tel opérateur est la différence entre l'opérateur
$$ g \longmapsto \Ric^{g} + (\delta^{g})^*B_{\bar g}(g) - \Ric^{\bar g} $$
agissant sur les métriques asymptotiquement hyperbolique et sa différentielle en la métrique 
$\bar g$ supposée polyhomogène et de même comportement dominant que $g$. Ce fait se constate facilement par exemple 
sur les expressions explicites de la courbure de Ricci et de l'opérateur de 
Bianchi données dans la section \ref{sec:trois}. Par exemple, on peut écrire tout terme de la 
forme (matricielle) $g^{-1}\partial^2g$ comme  (en notant $g=\bar g +r$)~:
\begin{align*}  
(\bar g + r)^{-1} \partial^2(\bar g + r) & = \bar g^{-1}\partial^2\bar g 
+ \left( \bar g^{-1}\partial^2r -  \bar g^{-1}r\bar g^{-1}\partial^2\bar g\right)\\
& \ \ \ \ +   \bar g^{-1}r\bar g^{-1}\partial^2r 
+ \bar g^{-1}\sum_{k=2}^{\infty}(-1)^k(r\bar g^{-1})^k\partial^2(\bar g+r).
\end{align*}
La différence avec la partie d'ordre inférieur ou égal à $1$ (vue comme agissant sur 
$r$) est donc
$$  \bar g^{-1}r\bar g^{-1}\partial^2r 
+ \bar g^{-1}\sum_{k=2}^{\infty}(-1)^k(r\bar g^{-1})^k\partial^2(\bar g+r) $$ 
et il est clair que cet opérateur vérifie toutes les hypothèses nécessaires pour être qualifié
d'opérateur polyhomogène au moins quadratique : pour tout poids $\delta$, il s'écrit bien comme la somme
d'une partie polynômiale à coefficients polyhomogènes en le $2$-jet de $r$ et d'un terme vivant dans un espace
de poids $\delta$ en la direction radiale (tronquer la somme et le développement de l'inverse de $\bar g$ à 
un ordre suffisamment élevé). Les autres termes apparaissant dans l'opérateur se traitent évidemment de la même façon.

\begin{prop}\label{prop:transverse}
Soit $k\geq 2$, $\alpha\in ]0,1[$ et $(\delta_1,\delta_2)$ vérifiant les conditions du lemme \ref{lemm:P}, 
et on suppose
que le premier poids critique supérieur $\mu_+$ de $L$ est supérieur ou égal à $1$. Si 
$u\in C^{k,\alpha}_{\delta_1,\delta_2}$ est solution de 
$$ Lu = \ell(u) + q(u) + f$$
où 
\begin{itemize}
\item $\ell$ est un opérateur différentiel linéaire à coefficients polyhomogènes 
qui appartient à l'algèbre $e^{-s}\mathcal{Q}_1$ dans le cas réel, $e^{-s/2}\mathcal{Q}_1$ dans le cas complexe~;
\item $q$ est un opérateur polyhomogène au moins quadratique~;
\item et $f$ est une section de $\dA_{\mu_++\eta,\delta_2}$ (pour un $ \eta>0$) modulo $C^{\infty}_{\infty,\delta_2}$,
\end{itemize} 
alors pour tout $\delta<1$ dans le cas réel et $\delta<1/2$ dans le cas complexe et pour tout opérateur différentiel 
linéaire $Q$ dans l'algèbre $\mathcal{Q}$,
\begin{itemize}
\item $Q(\nabw_H)^ku \in C^{0}_{\mu_++\delta,\delta_2}$ pour tout $k\in\NM$ dans le cas réel~;
\item $Q(\nabw_H)^k(\nabw_R)^{k'}u \in C^{0}_{\mu_++\delta,\delta_2}$ pour tous $(k,k')\in\NM^2$ dans 
le cas complexe.
\end{itemize}
\end{prop}

\begin{proof}
Le raisonnement diverge légèrement selon le caractère réel 
ou complexe de l'espace modèle, en raison de l'existence (ou non) d'une aniso­tropie dans les dérivées transverses. 
Nous nous contentons ici de la preuve dans le cas complexe, plus délicate, le cas réel s'en déduisant facilement. 

\smallskip

On commence par remarquer qu'en vertu du lemme de régularité, une section $u$ vérifiant les hypothèses de l'énoncé 
est en réalité dans $C^{\infty}_{\mu_+-\varepsilon,\delta_2}$ pour tout $\varepsilon>0$, ce qui justifie de 
s'intéresser à l'ensemble de ses dérivées. La suite de la preuve procède par récurrence, de manière analogue 
à ce qui est fait par exemple dans \cite{BiqHer05}~: le principe est de faire une récurrence sur $k+k'$, à l'intérieur
de laquelle s'insère une récurrence sur $k$.

{\flushleft\it Amorce ($k+k'=0$)}.
Il suffit de calculer
$$ L (\nabw_R u) = [ L , \nabw_R ] u + \nabw_R (Lu) .$$
Comme $ [ L , \nabw_R ] \in \mathcal{Q}$, le premier terme du membre de droite est dans $C^{\infty}_{\mu_+-\eps,\delta_2}$
pour tout $\eps>0$. Le second se décompose en 
$$ \nabw_R(Lu) = [\nabw_R, \ell] u + \ell(\nabw_R u) + \nabw_R(q(u)) + \nabw_Rf ;$$
on constate alors que 
\begin{enumerate}
\item comme $f$ est somme d'un terme polyhomogène fini et d'un terme décroissant très rapidement 
au voisinage de l'infini, $\nabw_R f$ est nécessairement 
dans $C^{\infty}_{\mu_+-\eps,\delta_2}$ (on notera ici qu'une dérivée selon $R$ ou $H$ d'un terme polyhomogène n'entraîne
pas de perte de poids---ce fait sera fréquemment utilisé dans la suite)~;
\item $q(u)$ est dans $C^{\infty}_{\mu_+-\eps+1,\delta_2}$ en raison de l'hypothèse sur le développement de Taylor
de la fraction rationnelle $q$ (on utilise ici que $\mu_+\geq 1$)~;
\item $[\nabw_R,\ell]$ est dans $\cQ$ en vertu des résultats de la fin de la section précédente, donc 
$[\nabw_R,\ell]u$ est dans $C^{\infty}_{\mu_+-\eps,\delta_2}$~;
\item\label{item:insuffisant} $\nabw_Ru$ est dans $C^{\infty}_{\mu_+-1-\eps,\delta_2}$ donc $\ell(\nabw_Ru)$ est dans 
$C^{\infty}_{\mu_+-1/2-\eps,\delta_2}$.
\end{enumerate} 
Autrement dit, $L (\nabw_R u)$ est dans $C^{\infty}_{\mu_+-\eps-1/2,\delta_2}$, d'où l'on déduit par le lemme de régularité que 
$\nabw_R u \in C^{\infty}_{\mu_+-\eps-1/2,\delta_2}$. Ce résultat peut être immédiatement réintroduit dans le
point (\ref{item:insuffisant})~: $\ell(\nabw_Ru)$ est alors dans $C^{\infty}_{\mu_+-\eps,\delta_2}$ et le lemme de régularité
fournit l'estimation souhaitée~:
$$\nabw_R u \in C^{\infty}_{\mu_+-\eps,\delta_2}.$$
En ce qui concerne la dérivée selon un vecteur $h$ dans $H$, on a de nouveau
$$ L (\nabw_h u) = [ L , \nabw_h ] u + \nabw_h (Lu) $$
et le terme le plus à droite est traité comme précédemment (noter qu'une dérivée selon $H$ ne fait perdre
qu'un poids égal à $1/2$, ce qui évite l'argument en deux temps utilisé plus haut pour la dérivée selon $R$). 
De plus, 
$$ [ L , \nabw_h ] u = -2\nabw_{e^{-s/2}Jh}\nabw_{e^{-s/2}R} u + Qu \ ; $$
la décroissance du premier terme provient donc de l'estimation obtenue quelques lignes
plus haut sur $\nabw_R u$ et toutes ses dérivées, tandis que le second terme ne pose
pas de problèmes. En appliquant une fois de plus le lemme de régularité, on obtient que
$e^{-s}\nabw_{R} u$ et $e^{-s/2}\nabw_{H} u$ sont dans $C^{\infty}_{\mu_++\delta,\delta_2}$ 
et donc que $Q u \in C^{\infty}_{\mu_++\delta,\delta_2}$ 
pour tout opérateur différentiel linéaire $Q$ dans $\mathcal{Q}$. Une relecture attentive du raisonnement que
nous venons de tenir montre que $\delta$ peut être pris
égal à n'importe quel réel plus petit que $1$ dans le cas réel et que $1/2$ dans le cas complexe.

\smallskip

{\flushleft\it Récurrence}. Donnons nous maintenant deux entiers $k,k'$ avec $k+k'>0$ et supposons que le résultat 
souhaité est connu pour tout couple $(k_1,k_1')$ avec $k_1+k_1'<k+k'$ ou avec $k_1+k_1'=k+k'$
et $k_1<k$. Prenons maintenant $k+k'+1$ champs de vecteurs $\xi_i$ sur le bord à l'infini, égaux
ou bien à un élément de $H$ ou à $R$ et notons $K=k+k'$. Comme précédemment
\begin{equation*}
\begin{split} 
L (\nabw_{\xi_1}\cdots\nabw_{\xi_{K+1}} u)  & = \ [ L ,\nabw_{\xi_1}\cdots\nabw_{\xi_{K+1}}  ] u \ + \ 
\nabw_{\xi_1}\cdots\nabw_{\xi_{K+1}} (Lu) \\
 & = \ \sum_{p=0}^{K+1} \, 
\nabw_{\xi_1}\cdots\nabw_{\xi_{p-1}} \, [ L , \nabw_{\xi_p} ]\,\nabw_{\xi_{p+1}}\cdots\nabw_{\xi_{K+1}} u \\
 & \ \ \ \ \ \ \ \ + \ \nabw_{\xi_1}\cdots\nabw_{\xi_{K+1}} (Lu).
\end{split}
\end{equation*}
Intéressons nous ici au premier terme du second membre~: pour tout $p$,
\begin{enumerate}
\item si $\xi_p=R$, $[L, \nabw_{\xi_p} ]\in \mathcal{Q}$ et on en conclut que
$$ \nabw_{\xi_1}\cdots\nabw_{\xi_{p-1}} \, [ L , \nabw_{\xi_p} ]\,\nabw_{\xi_{p+1}}\cdots\nabw_{\xi_{K+1}} u
\, \in \, C^{\infty}_{\mu_++\delta,\delta_2} $$ 
en utilisant conjointement l'hypothèse de récurrence avec $k_1=k$ et $k_1'=k_1'-1$ et celle avec
$k_1+k_1'=k+k'$ et $k_1=k_1-1$ (noter ici que les commutateurs entre dérivées transverses et éléments de $\mathcal{Q}$ 
ne font apparaître modulo $\mathcal{Q}$ que des termes contenant une dérivée de moins selon $H$ et le même nombre global 
de dérivées transverses, qui sont donc redevables de l'hypothèse de récurrence)~;
\item si $\xi_p\in H$, le commutateur $[L,\nabw_{\xi_p}]$ peut faire apparaître modulo $\mathcal{Q}$ un terme
en dérivées de $R$, qui est bien contrôlé d'après l'hypothèse de récurrence avec
$k_1+k_1'=k+k'$ et $k_1=k_1-1$.
\end{enumerate} 
En conclusion, chacun des termes dans la somme est dans $C^{\infty}_{\mu_++\delta,\delta_2}$, reste à traiter
le dernier terme du second membre. Celui se décompose à nouveau en une somme de termes faisant intervenir successivement
$f$, $q(u)$ et $\ell(u)$. Tous ces termes se traitent comme dans l'amorce (noter que l'hypothèse $\mu_+\geq 1$
joue de nouveau un rôle dans le contrôle du terme en $q(u)$), le terme en $\ell(\nabw_{\xi_1}\cdots\nabw_{\xi_{K+1}}u)$
étant une fois encore le plus embarrassant si $\xi_{K+1}=R$. Le lemme de régularité (appliqué éventuellement deux fois comme
dans l'amorce) assure alors que $\nabw_{\xi_1}\cdots\nabw_{\xi_{K+1}} u$ est dans $C^{\infty}_{\mu_+-\eps,\delta_2}$, et on en tire
immédiatement que toute expression de la forme $Q\nabw_{\xi_1}\cdots\nabw_{\xi_{K}} u$ est dans l'espace
$C^{\infty}_{\mu_++\delta,\delta_2}$ souhaité. 
\end{proof}

On notera que, dans la preuve, on aurait pu se contenter de prouver la
décroissance transverse pour les seules dérivées dans les directions
de $H$, celles dans les directions de $R$ s'en déduisant par
crochet. Cette remarque n'évite pas la nécessité d'estimer au moins
une dérivée selon $R$ puisque le commutateur de $L$ et de $\nabw_H$ en
fait apparaître. Nous avons préféré donner la preuve générale
ci-dessus car elle n'est en réalité pas plus longue.

%
%
%
%
%
%
%
%
%
%

\section{Développement polyhomogène des métriques d'Einstein asymptotiquement hyperboliques}
\label{sec:six}

Nous considérons maintenant une métrique d'Einstein lisse $g$ et
asymptotiquement hyperbolique sur une demi-boule
$$ M = \begin{cases}
    \{x_1^2+\cdots+x_n^2< 1, x_1> 0\} &\text{ dans le cas réel,}\\
    \{\big(x_1+\frac{x_3^2+\cdots+x_n^2}{2}\big)^2+x_2^2<1, x_1> 0\} &\text{ dans le cas complexe,}  
  \end{cases}
$$
de terme dominant à l'infini $g_0 = g_0(\gamma)$, avec $\gamma$ lisse, au sens où 
$$ g - g_0 \in C^{1,\alpha}_{\beta} $$
où $\alpha\in]0,1[$ et $\beta>0$. La valeur de $\beta>0$ sera toujours la même dans toute la suite,
choisie de façon à satisfaire les inégalités du lemme \ref{lemm:est-poids}. Par ailleurs,
nous noterons 
$$
\rho(x)=\begin{cases}
    \sqrt{x_1^2+\cdots+x_n^2} &\text{ dans le cas réel,}\\
    \sqrt{\big(x_1+\frac{x_3^2+\cdots+x_n^2}{2}\big)^2+x_2^2} &\text{ dans le cas complexe.}  
\end{cases}
$$
Quitte à restreindre la demi-boule, nous pouvons alors tronquer et recoller $g$
comme précédemment à la métrique hyperbolique modèle (voir de nouveau le début de
la section \ref{sec:quatre}), de telle sorte qu'elle s'étende en 
une métrique de courbure sectionnelle strictement négative sur le demi-espace $\{x_1>0\}$ tout entier, d'Einstein 
sauf sur un demi-anneau $\{x_1>0, \tfrac{1}{2}<\rho(x)<\tfrac{3}{4}\}$. 
La section \ref{sec:newtrois} produit alors une métrique conforme (que nous continuerons 
à noter $\gamma$, aucune ambiguïté n'étant à craindre) sur le bord à l'infini de la demi-boule 
$N=\{x_1>0, \rho(x)<\tfrac{1}{2}\}$.

Dans les coordonnées adaptées à cette nouvelle métrique, on a alors dans la demi-boule
$$ g - g_0(\gamma) \in C^{1,\alpha}_{\beta,\beta} .$$
Il est alors connu que l'on peut construire dans cette demi-boule une mé­tri­que asymptotiquement 
hyperbolique \emph{polyhomogène} qui est solution approchée des équations d'Einstein 
à un ordre élevé et dont le terme dominant est la métrique $g_0$. Cette 
construction remonte à Fefferman et Graham dans le cas réel \cite{FefGra85,Gra00} et est due à Seshadri 
dans le cas complexe \cite{Ses09}, voir aussi \cite{BiqHer05}. La méthode repose sur la résolution 
d'équations différentielles et s'applique donc tout aussi bien au cas où la métrique sur le bord
$\gamma$ est la métrique complète obtenue à la suite des constructions de la section \ref{sec:newtrois}. Dans ce dernier 
cas la solution approchée obtenue est polyhomogène dans les espaces à double poids définis dans la section
\ref{sec:quatre}. Plus précisément, en notant $\nu_0=n-1$ dans le cas réel et $\nu_0=m$ dans le cas 
complexe~:

\begin{lemm}
Il existe une métrique $\bar\phi$ sur la demi-boule vérifiant~:
\begin{enumerate} 
\item $\bar\phi-g_0 \in \dA_{1,1}$ dans le cas réel, $\dA_{1/2,1}$ dans le cas complexe~;
\item $\Ric^{\bar\phi} + c \bar\phi \in \dA_{\nu,1}$ modulo $C^{\infty}_{\infty,1}$ pour un certain $\nu>\nu_0$~;
\item les termes apparaissant dans le développement de $\bar\phi$ sont des termes exponentiels d'exposants tous 
entiers dans le cas réel et demi-entiers dans le cas complexe, sans termes polynômiaux à l'exception d'un éventuel 
terme en $se^{\nu_0s}$.
\end{enumerate}
\end{lemm}

Nous pouvons évidemment prolonger $\bar\phi$ sur
la boule tout entière en la recollant avec la métrique hyperbolique.
D'après la section \ref{sec:trois}, il est alors possible de trouver un difféomorphisme $\Phi$ de la
boule induisant l'identité au bord à l'infini de telle sorte que
la métrique tirée en arrière, que nous désignerons par $\phi=\Phi^*g$,
soit solution de
$$  B_{\bar\phi}(\phi) = 0 ,$$
sur la demi-boule $\{x_1>0, \rho(x)<\tfrac{1}{2}\}$.
Par le lemme \ref{lem:reg-locale}, on a alors sur la demi-boule 
$$ \phi - \bar\phi \in C^\infty_{\beta,\beta} $$
pour un certain $\beta >0$.
D'autre part, les équations impliquent immédiatement
$$ F(\phi) = \Ric^{\phi} + c \phi +(\delta^{\phi})^*B_{\bar\phi}(\phi) = 0.$$
avec $c = n-1$ dans le cas réel et $c=\tfrac{m+1}{2}$ dans le cas
complexe.

\smallskip

Puisque $\phi-\bar\phi\in C^{\infty}_{\beta,\beta}$, par le lemme
\ref{lemm:regularite}, on obtient
$$ \phi-\bar\phi \in C^{k,\alpha}_{\nu_0-\varepsilon, \beta} $$
pour tout $\varepsilon>0$, voir par exemple \cite[section 4]{BiqHer05}. 
On notera ici que le premier poids critique supérieur
$\mu_+$ de l'opérateur obtenu comme la différentielle de $F$ (en
n'importe quelle métrique asymptotiquement hyperbolique) n'est autre
que le nombre $\nu_0$ introduit plus haut~\cite{Biq00}, qui est toujours supérieur ou égal à $1$.

\smallskip

Le but de cette section est de prouver l'existence d'un développement polyhomogène de $\phi$. Pour des 
raisons techniques, il est commode d'intro­duire une version précisée de la polyhomogénéité. Nous 
notons à partir de maintenant $L$ la différentielle de $F$ en $g_0$~:

\begin{defi}
Un développement $L$-polyhomogène est un développement polyhomogène de la forme 
$$\sum a_{\sigma,\tau}\, r^{\sigma}\, e^{-\tau r}$$
dont toutes les puissances $\tau$ sont dans le monoïde (additif) $\NM_L$ 
engendré par les poids critiques supérieurs de $L$ et par $1$ dans le cas réel ou par $\tfrac12$
dans le cas complexe.
\end{defi}

L'intérêt de cette définition est que la suite des puissances $\tau$ qui apparaissent dans un
tel développement, une fois ordonnée, tend nécessairement vers l'infini. De manière un 
peu plus précise, on définit la suite de réels positifs 
$$ a_0 = 0, \quad \mu_+ + a_{k+1} = \min \{ \lambda \in \NM_L\ , \ \lambda > \mu_+ + a_{k} \}$$
où $\mu_+$ est le premier poids critique supérieur de $L=dF_{g_0}$ (\emph{cf}. plus haut). Autrement 
dit, la suite $(\mu_++a_k)$ est exactement la suite des puissances de $e^r$ qui apparaissent dans 
un développement $L$-polyhomogène. Par construction, on a donc $\mu_+ + a_k \in\NM_L$, 
$\lim_{k\to\infty} a_k = +\infty$, et $a_{k+1} \leq a_k + 1$ dans le cas réel, resp. $a_{k+1} \leq a_k + \tfrac12$
dans le cas complexe, tous faits utiles pour la suite.

Nous rappelons enfin que nous disposons d'une solution approchée et $L$-poly­ho­mo­gène $\bar{\phi}$,  
telle que $\bar\phi - g_0 \in \dA_{1,\beta}$ dans le cas réel, $\bar\phi-g_0 \in \dA_{1/2,\beta}$ dans le 
cas complexe et 
$$ F(\bar{\phi}) \in \mathcal{A}_{\nu,\beta} \ \textrm{ mod. } C^{\infty}_{\infty,\beta} ,$$
où $\nu>\mu_+=\mu_++a_0$. Le principal résultat technique de cet article est alors le suivant~:

\begin{theo}\label{th-polyhom} Si $\phi$ satisfait toutes les hypothèses précédentes, il existe une suite 
$(\phi_k)$ de solutions approchées $L$-\emph{polyho­mo­gènes} telle que
\begin{enumerate}
\item $\phi_k = \bar\phi + \psi_0 + \cdots + \psi_k$, avec $\psi_k\in\dA[\mu_+ + a_k]_{\beta}$
pour tout $k\in\NM$~;
\item pour tout $k\in\NM$, $F(\phi_k) \in \dA_{\mu_+ + a_{k + 1},\beta}$ modulo $C^{\infty}_{\infty,\beta}$~;
\item en notant $r_k = \phi - \phi_k$, pour tous $k$ et $p$ dans $\NM$, il existe $\delta>0$ 
tel que
$$(\nabemphtrans)^pr_k \in C^0_{\mu_+ + a_k + \delta,\beta} \ $$
\end{enumerate}
où $\nabemphtrans$ désigne toute dérivée dans une direction du bord à l'infini.
\end{theo}

Ce résultat entraîne immédiatement le suivant, qui constitue la version quantitative 
de notre théorème \ref{th:1} de l'introduction.

\begin{theo}\label{th:mainphg}
Soit une demi-boule 
$$ M = \begin{cases}
    \{x_1^2+\cdots+x_n^2< 1, x_1> 0\} &\text{ dans le cas réel,}\\
    \{x_1^2+x_2^2+\big(\frac{x_3^2+\cdots+x_n^2}{2}\big)^2<1, x_1> 0\} &\text{ dans le cas complexe,}  
  \end{cases}
$$ 
et $\gamma$ une métrique lisse sur $\partial_{\infty}M = \{ x_1=0\}\cap\bar M$, resp. $\eta$ une 
structure de contact lisse et $J$ une structure presque complexe lisse dans $H=\Ker\eta$ sur $\partial_{\infty}M$ telles que
$\gamma=d\eta(\cdot,J\cdot)$ soit définie positive. 

Si $g$ est une métrique d'Einstein asymptotiquement hyperbolique réelle, resp. complexe, au sens où  
$g-g_0(\gamma)\in C^{1,\alpha}_{\varepsilon}$ pour $\alpha\in]0,1[$ et $\varepsilon>0$, alors il 
existe une demi-boule $N$ incluse dans $M$, un difféomorphisme $\Phi$ de $\bar N$
induisant l'identité sur $\partial N\cup\partial_{\infty}N$, 
une suite de métriques asymptotiquement hyperboliques $g_k$ sur $N$, à développement polyhomogène fini,
et un couple $(\delta,\eta)$ de réels strictement positifs tels que
$$ \forall k\in\NM,\quad \Phi^*g - g_k \in C^{\infty}_{\mu_++a_k+\delta,\eta}(N)  .$$
De plus, la même estimation est valable pour toutes les dérivées transverses.
\end{theo}

{\flushleft\it Preuve du théorème \ref{th-polyhom}}. --
L'existence de la suite $(\phi_k)$ se démontre une nouvelle fois par récurrence. Profitons en
pour rappeler que la convention énoncée au début de la section \ref{sec:newcinq} est toujours en vigueur.

\subsection{Amorce} 
Elle est immédiate car on commence la récurrence en $k=-1$ en ajoutant aux définitions de l'énoncé $\phi_{-1}=\bar\phi$
(soit $\psi_{-1}=0$) et $a_{-1}=-1$. Le seul point à démontrer est l'estimation (3) sur les dérivées transverses du reste,
qui est une conséquence des résultats de la section précédente. De fait, un développement au premier ordre de 
l'opérateur $F$ au point $g_0$ donne
$$ 0 = F(\phi) = F(\bar{\phi}) + dF_{\bar{\phi}}(r_{-1}) + q_{-1}(r_{-1}). $$
Or
\begin{enumerate}
\item $F(\bar{\phi})\in\mathcal{A}_{\mu_++a_0+\delta,\beta}$ (pour un certain $\delta>0$) modulo 
$C^{\infty}_{\infty,\beta}$, par construction~;
\item $dF_{\bar{\phi}} = L + (dF_{\bar{\phi}} - L)$, où $\ell_{-1} := dF_{\bar{\phi}} - L$ est un
opérateur différentiel linéaire à coefficients polyhomogènes et
qui appartient à $e^{-s}\mathcal{Q}_1$ dans le cas réel ou $e^{-s/2}\mathcal{Q}$ dans le cas complexe~;
\item le reste $q_{-1}$ est un opérateur polyhomogène au moins quadratique au sens 
de la section précédente.
\end{enumerate}
On peut alors écrire
$$ L(r_{-1}) =  - F(\bar{\phi}) - \ell_{-1}(r_{-1}) - q_{-1}(r_{-1}) $$
et l'opérateur $L$ est redevable du résultat du lemme \ref{lemm:Pbis}
avec nos choix de métriques car $\tilde{A}+\mathcal{R}_0$ est positif, voir par exemple \cite{Biq00}.
Les résultats de la section \ref{sec:newcinq} fournissent alors l'existence d'un $\delta>0$ tel que
$$Q(\nabtrans)^p r_{-1} \in C^{0}_{\mu_+ + \delta,\beta}$$
pour tout $p\in\NM$ et tout $Q$ dans $\mathcal{Q}$. 

\smallskip

\subsection{Récurrence}
Supposons maintenant les $\phi_j$ construits pour $-1\leq j\leq k$ de telle sorte qu'ils
vérifient les propriétés (1--3) du théorème \ref{th-polyhom}. Le calcul-clé, que nous avons déjà utilisé, est le suivant~:
\begin{align*}
0 \ = \ F(\phi) & = \ F(\phi_k) \, + \, dF_{\phi_k}(r_k)  \, + \, q_{k} (r_k) \\
  & = \ F(\phi_k) \, + \, L(r_k) \, + \, (dF_{\phi_k}-L)(r_k)  \, + \, q_{k} (r_k) 
\end{align*}
où $q_k$ est l'opérateur obtenu à partir de $F$ en lui ôtant son linéarisé en $\phi_k$ et $L$ est, comme
précédemment, l'opérateur linéarisé en $g_0$, terme dominant de $\phi_k$. Il s'agit donc d'un opérateur polyhomogène  
au moins quadratique au sens de la section précédente. Ceci conduit à
\begin{equation*}
\begin{split}
\opI(r_k) \ = & - \, [F(\phi_k)]_{\mu_++a_{k+1}} \ - \ \left( F(\phi_k) - [F(\phi_k)]_{\mu_++a_{k+1}} 
\right) \\
 & \ \ - \Ltrans (r_k) \ - \ (dF_{\phi_k} - L )\, (r_k) \ - \ q_{k}\, (r_k).
\end{split}
\end{equation*}
Cette équation s'écrit de manière condensée sous la forme 
$$ \opI(r_k) \ = \ - \, [F(\phi_k)]_{\mu_++a_{k+1}} \ - \  e_{k+1} $$
où l'on a regroupé tous les termes du membre de droite sauf le premier dans $e_{k+1}$.
En utilisant les différentes hypothèses de la récurrence et le fait que $\mu_+\geq 1$, on voit de 
plus aisément que 
$$ e_{k+1} \in C^{0}_{\mu_+ + a_{k + 1} + \eps,\beta} $$
pour un $\eps>0$~: l'assertion (3) de l'hypothèse de récurrence et $\Ltrans\in\mathcal{Q}_0$ 
permettent de traiter le terme en $\Ltrans$, les autres estimations sont 
immédiates~; remarquer néanmoins une fois encore qu'on utilise $\mu_+\geq 1$ (hypothèse dont nous avons déjà noté
qu'elle est heureusement vérifiée dans les deux situations qui nous intéressent).
Pour définir $\psi_{k+1}$, il nous faut maintenant considérer différents cas suivant 
la position de $\mu_++a_k$ et $\mu_++a_{k+1}$ par rapport aux poids critiques supérieurs.

\smallskip

{\flushleft\it Cas 1 : $\mu_+ + a_k \geq \mu_+^{max}$}. On a donc $\mu_+ + a_k + \delta > \mu_+^{max} $ 
et le fait que $r_k$ soit dans $C^{0}_{\mu_++a_k+\delta,\beta}$ pour un $\delta>0$ entraîne immédiatement que 
$$ r_k \ = \ G_{\infty} ( [F(\phi_k)]_{\mu_++a_{k+1}} ) \ + \ G_{\infty} (e_{k+1}).$$
On pose alors 
$$\psi_{k+1} \ = \ G_{\infty} ( [F(\phi_k)]_{\mu_++a_{k+1}} ) \ \in \ \dA[\mu_++a_{k+1}]_\beta$$
où nous avons noté ici $G_{\infty}$ l'opérateur qui est égal sur chaque espace propre de $\widetilde{A} + \cR_0$ 
à l'opérateur scalaire $G_{\infty}$ défini plus haut avec le couple de poids 
$(\mu_-^{(i)},\mu_+^{(i)})$ correspondant à l'espace propre considéré.

Comme $\mu_+ + a_{k+1} > \mu_+^{max}$, tout échange entre dérivations transverses et intégrales entre $r$ et 
$+\infty$ est licite (convergence dominée évidente) donc $\psi_{k+1}$ est polyhomogène à coefficients 
\emph{lisses} et les coefficients vivent dans l'espace de poids $\beta$ sur les tranches. 
On définit alors $\phi_{k+1} = \phi_k + \psi_{k+1}$, et le reste des notations à l'avenant.
L'assertion (2) provient alors d'un développement de Taylor à l'ordre $1$ de $F(\phi_{k+1})$, analogue
à celui fait plus haut~:
\begin{equation*}
\begin{split}
F(\phi_{k+1} )  & = \ F (\phi_k + \psi_{k+1})\\
  & = \ F(\phi_k) \ + \ (dF_{\phi_k})(\psi_{k+1}) \ + \ q_k(\psi_{k+1}) \\ 
  & = \ F(\phi_k) \ + \  L(\psi_{k+1}) \ + \ (dF_{\phi_k} - L) (\psi_{k+1}) \ + 
      \ q_{k}(\psi_{k+1}),
\end{split}
\end{equation*}
où $q_k$ est l'opérateur (polyhomogène au moins quadratique au sens de la section précédente) 
obtenu en ôtant à $F$ sa valeur et son linéarisé en $\phi_k$. Ainsi, 
\begin{equation*}
\begin{split}
F(\phi_{k+1} )  & = \ [F(\phi_k)]_{\mu_+ + a_{k+1}} \ + \ \opI(\psi_{k+1}) \\
  & \ \ \ \ \ + \ F(\phi_k)- [F(\phi_k)]_{\mu_++a_{k+1}} \\
  & \ \ \ \ \ + \ \Ltrans(\psi_{k+1}) \ + \ (dF_{\phi_k} - L) (\psi_{k+1})
\ + \ q_{k}(\psi_{k+1}).
\end{split}
\end{equation*}
Dès lors,
\begin{enumerate}
\item la première ligne dans le membre de droite est nulle par construction de $\psi_{k+1}$~; 
\item comme $F(\phi_{k})$ est $L$-polyhomogène, le terme apparaissant 
à la deuxième ligne est dans $\dA_{\mu_+ + a_{k+2},\beta}$~: en effet, l'opérateur $F$ est
$$\phi \longmapsto \Ric^\phi + c\phi + (\delta^{\bar\phi})^*B_{\bar\phi}(\phi) $$
et le calcul explicite de $F(\phi_{k})$  (expression déjà étudiée précédemment) 
fait intervenir des dérivées radiales du terme polyhomogène $\phi_k$ (qui conservent les poids),
des dérivées transverses (qui produisent un gain d'un poids dans le cas réel ou d'un demi-poids dans le cas complexe), 
et des produits de ces mêmes termes, le tout modulo $C^{\infty}_{\infty,\beta}$. Au final, il apparaît un terme dominant,
nécessairement de poids $\mu_++a_{k+1}$ par hypothèse de récurrence, et des termes complémentaires dont les poids sont obtenus
par des sommes finies des différents $\mu_++a_i$ pour $i\leq k$ et d'entiers dans le cas réel, resp. de demi-entiers 
dans le cas complexe. Par définition de la suite $a_k$, tous ces poids sont supérieurs ou égaux au plus petit élément du
monoïde strictement supérieur à $\mu_++a_{k+1}$ soit $\mu_++a_{k+2}$~;
\item enfin, $\psi_{k+1}$ étant $L$-polyhomogène 
et dans $\dA[\mu_++a_{k+1}]_\beta$, tous les autres termes sont dans $\dA_{\mu_++a_{k+1}+1}$ dans le cas
réel et $\dA_{\mu_++a_{k+1}+1/2}$ dans le cas complexe car $dF_{\phi_k}-L$ 
est dans $e^{-s}\mathcal{Q}_1$ dans le cas réel ou  $e^{-s/2}\mathcal{Q}_1$ dans le cas complexe, $q_k$
préserve les poids et $\mu_+\geq 1$. 
\end{enumerate}
Dans tous les cas, 
$$F(\phi_{k+1})\in \dA_{\mu_++a_{k+2},\beta}  \ \textrm{ mod. } C^{\infty}_{\infty,\beta} .$$ 
Démontrons enfin l'assertion (3), en étudiant d'un peu plus près 
$$ e_{k+1} = \left( F(\phi_k) - [F(\phi_k)]_{\mu_++a_{k+1}} \right) \ + \ \Ltrans (r_k) \ + 
\ (dF_{\phi_k} - L )\, (r_k) \ + \ q_{k}\, (r_k). $$ 
On constate alors que
\begin{enumerate}
\item par $L$-polyhomogénéité, toutes les $\nabtrans$-dérivées du premier terme de $e_{k+1}$
sont dans $\dA_{\mu_++a_{k+2},\beta}$ modulo $C^{\infty}_{\infty,\beta}$~;
\item par hypothèse de récurrence sur $r_k$, plus précisément l'assertion (3), toutes celles du deuxième 
terme sont dans $C^0_{\mu_++a_{k}+1+\eta,\beta}$ dans le cas réel, resp. $C^0_{\mu_++a_{k}+1/2+\eps,\beta}$ 
dans le cas complexe, 
pour un certain $\eps>0$ (se souvenir que $\Ltrans$ est dans $\mathcal{Q}_0$ d'où le gain d'un poids, resp. demi-poids... 
--éventuellement avec une correction $-\eta$ à cause de la possible apparition de termes polynômiaux)~;
\item l'assertion (3) sur $r_k$ 
entraîne que toutes les dérivées transverses de $(dF_{\phi_k} - L )\, (r_k)$ sont 
dans le même espace $C^0_{\mu_++a_{k}+1+\eps,\beta}$ dans le cas réel ou $C^0_{\mu_++a_{k}+1/2+\eps,\beta}$ dans le cas complexe 
car $dF_{\phi_k} - L$ est à coefficients polyhomogènes et vit donc dans $e^{-s}\mathcal{Q}_1$ 
ou $e^{-s/2}\mathcal{Q}_1$, d'où de même un gain de presque un ou un demi poids~;
\item enfin, les dérivées transverses de $q_{k}\, (r_k,r_k)$ sont toujours dans ce même espace en
utilisant encore une fois l'assertion (3) pour $r_k$, les hypothèses sur $q_k$ et que $\mu_+\geq 1$.
\end{enumerate}
On en conclut donc que $(\nabtrans)^p e_{k+1}\in C^{0}_{\mu_++a_{k+1}+\delta,\beta}$, pour un $\delta>0$.
En particulier, $G_{\infty}$ commute aux dérivées transverses car $\mu_++a_{k+1}+\delta$ est 
strictement supérieur à $\mu_+^{max}$. Se souvenant que $r_{k+1} = G_{\infty}(e_{k+1})$ 
ici, on a 
$$ \forall p \in\NM,\quad (\nabtrans)^p r_{k+1} \in C^0_{\mu_++a_{k+1}+\delta,\beta}\ ,$$
ce qui est le résultat voulu. 

\smallskip 

{\flushleft\it Cas 2 : $\mu_+ + a_k <  \mu_+^{max}$}. Ce cas ne se rencontre qu'un nombre fini de fois
puisque les poids critiques supérieurs sont en nombre fini et que la suite $(a_k)$ tend
vers l'infini par construction (c'est ici que la $L$-polyhomogénéité et l'ensemble $\NM_L$ jouent un rôle
commode). Il se divise de plus en deux sous-cas suivant que $[\mu_+ + a_k + \delta,\mu_+ + a_{k+1}]$ contient ou 
non un poids critique (supérieur).

\smallskip

{\flushleft\it Sous-cas 2.1 : $[\mu_+ + a_k + \delta,\mu_+ + a_{k+1}]$ ne contient pas de $\mu_+^{(i)}$}.
On a alors 
$$ r_k \ = \ G( [F(\phi_k)]_{\mu_++a_{k+1}} ) \ + \ G(e_{k+1}) \ + \ A(s)$$
où, sur le $i$-ème espace propre de $A +\cR_0$, 
$$G = 
\begin{cases}
G_0  & \textrm{ si } \mu_++a_{k+1}<\mu_+^{(i)}\\
G_{\infty} & \textrm{ sinon}
\end{cases}$$ 
ces deux opérateurs étant définis avec le couple de poids $(\mu_-^{(i)},\mu_+^{(i)})$ correspondant à l'espace propre, 
et le terme $A(s)$ est une somme finie de la forme 
$$ \sum_i A_i\, e^{-\mu_+^{(i)}s} $$ 
où les $A_i$ sont des sections sur le bord à l'infini et $i$ parcourt seulement les espaces propres 
tels que $\mu_++a_{k+1}<\mu_+^{(i)}$.
On peut alors définir
$$ \psi_{k+1} = G([F(\phi_k)]_{\mu_++a_{k+1}} )  \ \in \ \dA[\mu_++a_{k+1}]_\beta ,$$ 
puis $\phi_{k+1}$, $r_{k+1}$, etc. comme précédemment.
 
Vérifions d'abord que $F(\phi_{k+1})$ est dans $\dA_{\mu_++a_{k+2},\beta}$ comme souhaité. 
On commence par remarquer que $\psi_{k+1}$ est une somme de termes de deux types~: 
\begin{itemize}
\item[-] des intégrales entre $r$ et $+\infty$ là où 
$\mu_++a_{k+1}>\mu_+^{(i)}$, donc où tous les échanges entre intégrales et dérivations transverses 
sont licites~;
\item[-] des intégrales sur le segment $[r_0,r]$ où ces mêmes dérivations ne posent évidemment pas de
problèmes.
\end{itemize}
On écrit alors de nouveau
\begin{equation*}
\begin{split}
F(\phi_{k+1} )  & = \ [F(\phi_k)]_{\mu_+ + a_{k+1}} \ + \ \opI(\psi_{k+1}) \\
  & \ \ \ \ \ + \ F(\phi_k)- [F(\phi_k)]_{\mu_++a_{k+1}} \\
  & \ \ \ \ \ + \ \Ltrans(\psi_{k+1}) \ + \ (dF_{\phi_k} - L) (\psi_{k+1})
\ + \ q_{k}(\psi_{k+1})
\end{split}
\end{equation*}
et on constate que les mêmes arguments que ceux décrits dans le cas 1 fournissent que
$$F(\phi_{k+1})\in \dA_{\mu_++a_{k+2},\beta}.$$ 
Une étude plus poussée du terme $A(s)$ est nécessaire pour démontrer l'asser­tion (3). Il s'agit 
d'un élément du noyau de l'opérateur indiciel, qui ne possède pour l'instant aucune régularité particulière 
en les directions transverses, et la décroissance de ses éventuelles dérivées transverses (si elles existent)
n'est pas contrôlée. La méthode permettant d'obtenir cette régularité est identique dans les deux sous-cas 
2.1 et 2.2~; elle entraîne immédiatement, comme nous le verrons, que l'assertion (3) de l'hypothèse de 
récurrence est vérifiée au cran $k+1$. Nous repoussons donc au paragraphe \ref{subsec7.3} ci-dessous 
la preuve de ces faits et nous nous intéressons d'abord au second sous-cas.
 
\smallskip

{\flushleft\it Sous-cas 2.2 : $[\mu_+ + a_{k} + \delta,\mu_+ + a_{k+1}]$ contient un poids critique 
supérieur}, qui est nécessairement égal à $\mu_++a_{k+1}$ par définition de la suite $(a_k)$.
On peut alors écrire comme précédemment
$$ r_k \ = \ G( [F(\phi_k)]_{\mu_++a_{k+1}} ) \ + \ G(e_{k+1}) \ + \ A_0 e^{-(\mu_++a_{k+1})s} \ + \ A(s)$$
où $A_0$ est une section sur le bord à l'infini (pour l'instant sans régularité)~: comme précédemment, 
on a sur le $i$-ème espace propre de $A +\cR_0$
$$G = 
\begin{cases}
G_0  & \textrm{ si } \mu_++k+\eps\leq \mu_+^{(i)} \\
G_{\infty} & \textrm{ sinon} 
\end{cases}$$ 
et le terme $A(s)$ est défini comme dans le sous-cas 2.1. On pose alors 
$$ \psi_{k+1} = G([F(\phi_k)]_{\mu_++a_{k+1}} ) \ + \ A_0 e^{-(\mu_++a_{k+1})s} $$
(attention, ici le premier des deux termes de droite peut être nul), etc.
L'étape suivante est d'obtenir la régularité transverse de $A_0$ et des termes intervenant dans $A(s)$.
Comme dans le sous-cas précédent, nous repoussons cette preuve au paragraphe \ref{subsec7.3} ci-dessous~; 
l'assertion (3) s'en déduira immédiatement. Une fois cette étape
réalisée, nous pouvons vérifier que $F(\phi_{k+1})$ est dans $\dA_{\mu_++k+2,\beta}$, ce qui s'obtient 
en reprenant mot pour mot les raisonnement déjà tenus dans le cas 1 et le sous-cas 2.1. 
On notera néanmoins que des dérivées de $\psi_{k+1}$ interviennent 
dans les arguments menant à la décroissance de $F(\phi_{k+1})$~; comme le terme en 
$A_0 e^{-(\mu_++a_{k+1})s}$ est partie prenante de $\psi_{k+1}$, il est impératif d'avoir obtenu 
la régularité transverse auparavant pour conclure.

\smallskip

\subsection{Régularité transverse et assertion (3) dans les sous-cas 2.1 et 2.2}\label{subsec7.3}
Comme plus haut, nous décomposons $r_{k+1}$ et $\psi_{k+1}$ sur les espaces propres de $\tilde A +\cR_0$, de poids
critiques associés $(\mu_-^{(i)},\mu_+^{(i)})$. Les composantes d'une section $u$ seront alors notées
$u^{(i)}$. 

Sur les espaces propres où $\mu_++a_{k}\geq\mu_+^{(i)}$,
$$ \psi_{k+1}^{(i)} = G_{\infty}\left( [F(\phi_k)]_{\mu_++a_{k+1}}^{(i)} \right) $$
est donc polyhomogène à coefficients lisses, 
$$ r_{k+1}^{(i)} = G_{\infty}(e_{k+1}^{(i)}) $$
et la régularité transverse et les contrôles des dérivées transverses sont donc obtenus comme dans le 
cas 1.

Sur les espaces propres où $\mu_++a_{k+1}<\mu_+^{(i)}$, 
$$ \psi_{k+1}^{(i)} = G_{0}\left( [F(\phi_k)]_{\mu_++a_{k+1}}^{(i)} \right) $$
est polyhomogène à coefficients lisses (l'intégrale potentiellement problématique est une fois encore une
intégrale sur un segment). De plus, 
$$ r_{k+1}^{(i)} = r_k^{(i)} - \psi_{k+1}^{(i)} =  G_{0}(e_{k+1}^{(i)}) + A_i e^{-\mu_+^{(i)}s} $$
où $A_i$ est un terme ne dépendant que des variables transverses et qui est pour l'instant sans régularité particulière.
Considérons cependant cette expression sur la tranche $s=s_0$~: $r_k^{(i)}$ est lisse en les variables transverses par 
hypothèse de récurrence, $\psi_{k+1}^{(i)}$ l'est aussi comme on vient de le remarquer, le terme en 
$e^{-\mu_-^{(i)}s}$ dans $G_{0}(e_{k+1}^{(i)})$ l'est également, et le second terme intégral est nul
puisque $s=s_0$. Donc $A_i$ est nécessairement lisse en les directions transverses (un contrôle de la décroissance
des dérivées transverses sera obtenu un peu plus loin).

Il nous reste enfin à étudier dans le sous-cas 2.2 l'espace propre dont le poids critique supérieur est égal à 
$\mu_++a_{k+1}$. Nous notons ce poids critique $\mu_+^{(0)}$ par facilité, l'espace propre correspondant 
étant associé à l'indice $i=0$. On a alors
$$ r_{k+1}^{(0)} = G_{0}(e_{k+1}^{(0)}) $$
et ce terme est donc lisse en les directions transverses (une fois encore, l'intégrale potentiellement 
problématique est sur un segment). De plus,
$$ \psi_{k+1}^{(0)} = r_{k+1}^{(0)} - r_k^{(i)} = G_{0}\left( [F(\phi_k)]_{\mu_++a_{k+1}}^{(0)} \right) 
+ A_0 e^{-\mu_+^{(0)}s} $$
et on obtient la régularité de $A_0$ comme dans le cas précédent en fixant $s=s_0$, car $r_k^{(0)}$ est
là encore lisse en les directions transverses par hypothèse de récurrence.

\smallskip

Démontrons maintenant les estimées attendues sur les dérivées transverses des $r_{k+1}^{(i)}$ dans les 
deux derniers cas. Elles sont immédiates pour les termes en les $A_i$ puisqu'ils sont polyhomogènes, 
vivent dans $\dA_{\mu_+^{(i)},\beta}$ avec $\mu_+^{(i)}>\mu_++a_{k+1}$, et sont lisses en les directions 
transverses~: toute dérivée transverse est donc de même poids. Il ne reste donc qu'à estimer les
dérivées transverses 
$$ (\nabtrans)^p \, G_{0}(e_{k+1}^{(i)}) = G_{0}\left( (\nabtrans)^p \,e_{k+1}^{(i)} \right) \ ,$$
où $i$ est éventuellement égal à $0$. L'échange des dérivées et des intégrales étant autorisé puisque
le terme potentiellement problématique est en réalité une fois encore une intégrale sur un segment,
les estimations s'obtiennent alors en remarquant que pour toute fonction $f$ de l'espace à simple poids $C^{0}_{\alpha}$,
$$ \left| e^{-\theta s}\int_{+\infty}^{s} e^{\theta \varsigma} f d\varsigma \right| \leq C e^{-\alpha s}$$
si $\theta\neq \alpha$, et
$$ \left| e^{-\alpha s} \int_{s_0}^s e^{\alpha \varsigma} f d\varsigma \right| \leq C s e^{\alpha s} .$$

%
%
%
%
%
%
%
%

\section{Continuation unique dans le cas complexe}\label{sec:huit}

La démonstration de la continuation unique est analogue à celle dans
le cas réel \cite{Biq08}, donc nous ne donnerons pas tous les
détails. Il y a trois ingrédients :
\begin{enumerate}
\item l'existence d'une « équation spéciale » du bord, qui permet de
  mettre les métriques dans une jauge géodésique ;
\item la coïncidence à un ordre infini des deux métriques $g_1$ et
  $g_2$ dans la jauge géodésique, qui résulte de la polyhomogénéïté et
  du calcul des poids critiques supérieurs de la linéarisation $L$ ;
\item la coïncidence de $g_1$ et $g_2$ grâce aux estimations de
  Carleman, utilisées via la méthode de \cite{Biq08}.
\end{enumerate}

Commençons donc par (1). Une équation spéciale du bord
est une fonction $\upsilon$, s'annulant au bord $\{u=0\}$, telle que
$|d\upsilon|_{\upsilon^2g}=1$, ce qui s'écrit encore $|\frac{d\upsilon}\upsilon|_g=1$. On cherche
$\upsilon$ sous la forme $\upsilon=e^{2\phi}u$, d'où vient l'équation
\begin{equation}
  \label{eq:62}
  \big\langle \frac {d\phi}{\sqrt u} , \frac{du}u \big\rangle_g
 + \frac{|d\phi|_g^2}{\sqrt u}
 = \frac 1{4\sqrt u}\Big(1 - \big|\frac{du}u\big|_g^2\Big), \quad \phi|_{u=0}=0.
\end{equation}
Compte tenu du contrôle de la métrique obtenu au lemme
\ref{lem:modele-local}, le second membre est borné, en réalité
$O(v)$. D'un autre côté, le membre de gauche est une équation aux
dérivées partielles d'ordre 1 sur $\phi$, se réduisant à
$\frac{\partial\phi}{\partial\sqrt u}$ sur le bord, donc l'hypersurface ${\sqrt u=0}$
n'est pas caractéristique. (Le lecteur aura remarqué que pour résoudre
l'équation, il faut considérer $\sqrt u$ plutôt que $u$ comme équation
lisse du bord). Compte tenu de la forme de la métrique, aucune
singularité ne provient du bord $\{v=0\}$, et on déduit qu'existe, au
voisinage du bord, une unique solution $\phi$ de l'équation
(\ref{eq:62}), satisfaisant
$$ \phi=O(\sqrt u v). $$
(On peut borner aussi les dérivées de $\phi$).

En utilisant les niveaux de la nouvelle fonction $\upsilon$, la métrique est
exprimée sous la forme 
\begin{equation}
g=\frac{d\upsilon^2}{\upsilon^2}+g_\upsilon,\label{eq:68}
\end{equation}
où $g_\upsilon$ est une famille de métriques sur le bord à l'infini,
satisfaisant les mêmes conditions asymptotiques que dans le lemme
\ref{lem:modele-local}.  En outre, la seconde forme fondamentale $\II$
des tranches $\{\upsilon=\mathrm{cst}.\}$ est uniformément bornée, ainsi que ses
dérivées, et a asymptotiquement pour valeurs propres $1$ et $\frac 14$ :
\begin{equation}
  \label{eq:67}
  |\nabla^k\II| \leq c_k , \quad \II \longrightarrow
  \begin{pmatrix}
    1 & & \\ & \tfrac 14 & \\ & & \ddots
  \end{pmatrix}.
\end{equation}

Enfin, si deux métriques $g_1$ et $g_2$ coïncident à l'ordre
$o(e^{-(m+1)r})=o(x^{m+1})$, alors après mise en jauge géodésique il
est clair que la coïncidence persiste, avec
$$ |g_1-g_2| = o((\upsilon v^2)^{m+1}). $$

Passons à (2). L'existence d'un développement polyhomogène de la
métrique dans une jauge de Bianchi a pour conséquence que la métrique
garde un développement polyhomogène dans la jauge géodésique obtenue
grâce à une équation spéciale. En outre, les poids critiques pour ce
développement sont nécessairement inclus dans ceux du développement en
jauge de Bianchi, c'est-à-dire sont des poids critiques supérieurs de
la linéarisation $L$. Un calcul facile basé sur les formules de
théorie des représentations de \cite{Biq00} donnne les valeurs $m$,
$\frac12(m+\sqrt{m^2+8})$, $\frac12(m+\sqrt{m^2+2m+5})$ et $m+1$, donc
le plus grand poids critique est $m+1$. La coïncidence jusqu'à l'ordre
$o(\upsilon^{m+1})$ implique donc la coïncidence jusqu'à l'ordre infini:
$O(\upsilon^\infty)$. (Une autre possibilité pour comprendre les poids critiques
en jauge géodésique consiste à utiliser le calcul par Seshadri
\cite[§4]{Ses09} de l'équation d'Einstein en jauge géodésique).

Enfin, examinons la dernière étape (3) : une fois qu'on a la
coïncidence à un ordre infini en jauge géodésique, et les estimations
de Carleman, la démonstration de \cite{Biq08} s'applique verbatim, et
nous ne la répéterons pas. Le dernier pré-requis pour la preuve est
donc l'estimation de Carleman : si $\Delta$ est un opérateur elliptique
géométrique du second ordre, agissant sur les sections d'un fibré, de
type $\nabla^*\nabla+$termes d'ordre inférieur bornés, alors il existe $\upsilon_0$ et
$C$ tel que, pour toute section $f$ à support compact dans
$\{0<\upsilon<\upsilon_0\}$, qui est $O(v)$ ainsi que ses dérivées, et pour tout $\lambda\gg
0$, on ait
\begin{equation}
  \label{eq:66}
  \int_{0<\upsilon<\upsilon_0} |\Delta f|^2 \upsilon^{-\lambda} \vol_g
\geq C \int_{0<\upsilon<\upsilon_0} \big( \lambda^3 |f|^2 + \lambda |\nabla f|^2 + \lambda^{-1} |\nabla^2f|^2 \big)
\upsilon^{-\lambda} \vol_g .
\end{equation}
L'estimation de Carleman dans ce contexte est connue depuis longtemps,
on en trouvera une démonstration sous les hypothèses géométriques
(\ref{eq:68}) et (\ref{eq:67}) dans l'article \cite{VasWun05}. La
différence ici est que les tranches $\{\upsilon=\mathrm{cst}\}$ sont non
compactes, mais comme elles sont complètes cela n'empêche pas les
intégrations par parties, possibles vu l'hypothèse sur $f$ quand $v\to0$.

%
%
%
%
%
%
%
%

\bibliographystyle{smfplain}
\bibliography{biblio,biquard,phg}

\end{document}